\documentclass[12pt, reqno]{amsart}
\pdfoutput=1
\usepackage{mathpreamble}

\title{The limit set of non-orientable mapping class groups}
\author{Sayantan Khan}
\address{Department of Mathematics, University of Michigan, Ann Arbor, MI}
\email{\href{mailto:saykhan@umich.edu}{saykhan@umich.edu}}
\thanks{The author was partially supported by NSF Grant DMS 1856155.}
\urladdr{\url{https://www-personal.umich.edu/~saykhan/}}


\date{\today}

\begin{document}

\begin{abstract}
  We provide evidence both for and against a conjectural analogy between geometrically finite infinite covolume Fuchsian groups and the mapping class group of compact non-orientable surfaces.
  In the positive direction, we show the complement of the limit set is open and dense.
  Moreover, we show that the limit set of the mapping class group contains the set of uniquely ergodic foliations and is contained in the set of all projective measured foliations not containing any one-sided leaves, establishing large parts of a conjecture of Gendulphe.
  In the negative direction, we show that a conjectured convex core is not even quasi-convex, in contrast with the geometrically finite setting.
\end{abstract}
\maketitle

% \tableofcontents

\section{Introduction}
\label{sec:introduction-v2}

The moduli space $\mg(\no_g)$ of compact \emph{non-orientable} hyperbolic surfaces of genus $g$ is conjectured to have similarities to infinite volume geometrically finite manifolds (in a manner similar to how moduli spaces of compact orientable surfaces have properties similar to finite volume hyperbolic manifolds).
The main results suggesting the analogy between moduli spaces of non-orientable surfaces and infinite volume geometrically finite manifolds are due to Norbury and Gendulphe.

\begin{itemize}
\item The $\mg(\no_g)$ has infinite Teichm\"uller volume \cite[Theorem 17.1]{gendulphe_whats_2017}.
  While the associated Teichm\"uller space does not have a Weil-Peterson volume form, it has an analogous volume form with respect to which the moduli space has infinite volume as well (see \cite{norbury2008lengths}).
\item The action of the mapping class group $\mcg(\no_g)$ on the Thurston boundary is not minimal (Proposition 8.9 in \cite{gendulphe_whats_2017}).
\item The Teichm\"uller geodesic flow is not topologically transitive, and thus not ergodic with respect to any Borel measure with full support \cite[Proposition 17.5]{gendulphe_whats_2017}.
\item There exists an $\mcg(\no_g)$-equivariant finite covolume deformation retract of $\teich(\no_g)$.
\end{itemize}

We extend this analogy further, by showing that the limit set of $\mcg(\no_g)$ is contained in the complement of a full measure dense open set.
\begingroup
\def\thetheorem{\ref{cor:geolimset}}
\begin{theorem}
  The limit set of $\mcg(\no_g)$ is contained in the complement of $\pmf^-(\no_g)$.
\end{theorem}
\addtocounter{theorem}{-1}
\endgroup
Here $\pmf^-(\no_g)$ is the set of all projective measured foliations that have one-sided compact leaf.
The fact that such foliations form a full measure dense open subset is classical, due to Danthony-Nogueira (see \cite{ASENS_1990_4_23_3_469_0}).
This is analogous to limit sets of infinite volume geometrically finite groups, where the complement of the limit set is a full measure open set as well.

In \cite{gendulphe_whats_2017}, Gendulphe constructed a retract of $\teich(\no_g)$ to $\systole$, the set of points in the Teichm\"uller space that have no one-sided curves shorter than $\varepsilon$, and showed that it has finite covolume.
They also asked the following question about $\systole$.
\begin{unquestion}[Question 19.1 of \cite{gendulphe_whats_2017}]
  Is $\systole$ quasi-convex with respect to the Teichmüller metric?
\end{unquestion}
We show that ${\systole}$ is not quasi-convex, answering the above question.
\begingroup
\def\thetheorem{\ref{thm:qc-fail}}
\begin{theorem}
  For all $\varepsilon > 0$, and all $D > 0$, there exists a Teichm\"uller geodesic segment whose endpoints lie in ${\systole}$ such that some point in the interior of the geodesic is more than distance $D$ from $\systole$.
\end{theorem}
\addtocounter{theorem}{-1}
\endgroup

Since $\systole$ is an $\mcg(\no_g)$-invariant subset of $\teich(\no_g)$, the intersection of its closure with the boundary must also be $\mcg(\no_g)$-invariant, and therefore contain the limit set of $\mcg(\no_g)$.
This suggests that if we want long geodesic segments that start and end in $\systole$, we must look for Teichmüller geodesics that have their expanding and contracting foliations in the limit set.
Conjecture 9.1 of \cite{gendulphe_whats_2017} states that the limit set should exactly be the complement of $\pmf^-(\no_g)$, the set of projective measured foliations that do not contain any one-sided leaves (denoted $\pmf^+(\no_g)$).
We prove a result that is slightly weaker than the conjecture.
\begingroup
\def\thetheorem{\ref{thm:rational-approximation}}
\begin{theorem}
  A foliation $\lambda \in \pmf^+(\no_g)$ is in the limit set of $\mcg(\no_g)$ if all the minimal components $\lambda_j$ of $\lambda$ satisfy one of the following criteria.
  \begin{enumerate}[(i)]
  \item $\lambda_j$ is periodic.
  \item $\lambda_j$ is ergodic and orientable, i.e. all leaves exiting one side of a transverse arc always come back from the other side.
  \item $\lambda_j$ is uniquely ergodic.
  \end{enumerate}
  Furthermore, if $\lambda_j$ is minimal, but not uniquely ergodic, there exists some other foliation $\lambda_j^{\prime}$ supported on the same topological foliation as $\lambda_j$ which is in the limit set.
\end{theorem}
\addtocounter{theorem}{-1}
\endgroup

With this description of the limit set, we prove \autoref{thm:qc-fail} by constructing a family of Teichmüller geodesics whose expanding and contracting foliations are of the kind described by \autoref{thm:rational-approximation}, and showing that some point in the interior of the geodesic segment is arbitrarily far from $\systole$.

We now exhibit two additional contexts in which one is naturally led to consider the limit set of $\mcg(\no_g)$.

\subsection*{Counting simple closed curves}

In the orientable setting, the number of simple closed geodesics of length less than $L$ grows like a polynomial of degree $6g-6$, which is precisely the dimension of the limit set of the mapping class group: in this case, that happens to be all of $\pmf$.
In the non-orientable setting, Gendulphe showed that the growth rate of the simple closed geodesics of length less than $L$ is smaller than $L^{\dim(\pmf(\no_g))}$, and one might conjecture that the growth rate is $L^h$, where $h$ is the Hausdorff dimension of the limit set.
Mirzakhani, in \autocite{mirzakhani2008growth}, obtained precise asymptotics for the counting function in the orientable case by essentially proving equidistribution (with respect to Thurston measure) of $\mcg(\os_{g})$-orbits in $\pmf(\os_g)$.
The same problem for non-orientable surfaces is posed in \cite[Problem 9.2]{wright2020tour}: to make the above techniques work in this setting, we need an ergodic measure supported on sets
minimal with respect to the $\mcg$ action, e.g. $\overline{\pmf^+(\no_d)}$.
One way to construct such a measure would be to replicate the construction of Patterson-Sullivan measures for geometrically finite manifolds, which brings us back
to the analogy between $\mg(\no_d)$ and infinite volume geometrically finite manifolds.
In the case of $\no_{1,3}$ (i.e. the real projective plane with $3$ punctures), Gamburd, Magee, and Ronan have proved a counting result for simple closed curves by constructing a conformal measure of non-integer Hausdorff dimension on the limit set (\cite[Theorem 10]{10.4007/annals.2019.190.3.2}), and then using that conformal measure to count simple closed curves (\cite[Theorem 2]{10.1093/imrn/rny112}).

\subsection*{Interval exchange transformations with flips}

Teichm\"uller spaces of non-orientable surfaces also show up in the context of \emph{interval exchange transformations with flips}.
The dynamics of interval exchange transformations are closely related to the dynamics of horizontal/vertical flow on an associated quadratic differential, which is related to the geodesic flow on the Teichm\"uller surface via Masur's criterion (a version of which holds in the non-orientable setting as well).
IETs with flips do not have very good recurrence properties: in fact, almost all of them (with respect to the Lebesgue measure) have a periodic point (see \cite{nogueira_1989}) and the set of minimal IETs with flips have a lower Hausdorff dimension (see \cite{skripchenko2018hausdorff}).
To understand the IETs which are uniquely ergodic, one is naturally led to determine which ``quadratic differentials'' on non-orientable surfaces are recurrent.
A necessary but not sufficient condition for recurrence of a Teichm\"uller geodesic is that its forward and backward limit points lie in the limit set.
From this perspective, Theorems \ref{thm:rational-approximation} and \ref{cor:geolimset} can
be seen as a statement about the closure of the recurrent set.
Constructing a measure supported on the closure of the recurrent set can be then used to answer questions about uniquely ergodic IETs with flips.

\subsection*{Another paper on mapping class group orbit closures}
A few days after the first version of this paper was posted on arXiv, the author was notified of another recent paper by Erlandsson, Gendulphe, Pasquinelli, and Souto \cite{erlandsson2021mapping} which proves the Conjecture 9.1 of \cite{gendulphe_whats_2017}, i.e. a stronger version of \autoref{thm:rational-approximation}.
The techniques they use are significantly different, relying on careful analysis of train track charts carrying various measured laminations.
Using this result, they show minimal invariant subset of $\pmf(\no_g)$, and the limit set of $\mcg(\no_g)$ are both equal to $\pmf^+(\no_g)$.

Their methods are stronger than the ones in this paper, because while they analyze the train track charts carrying the measured foliations, we study the dynamics of foliations by studying the dynamics of the first return map on a transverse interval.
While this reduction makes the analysis significantly simpler, and works for most foliations (like uniquely ergodic foliations), it does not work for all foliations.
In particular, for certain foliations, the dynamics of the first return map do not fully capture the dynamics of the foliation.
% The authors of the other paper get around this limitation by directly analyzing the train tracks associated to the foliation, which fully captures their dynamics.

While this paper was being written, neither the author, nor the authors of \cite{erlandsson2021mapping} were aware of each others' work.

\subsection*{Organization of the paper}
Section \ref{sec:backgr-meas-foli} contains the background on non-orientable surfaces and measured foliations, and section \ref{sec:backgr-limit-sets} contains the background on limit sets of mapping class subgroups.
These sections can be skipped and later referred to if some notation or definition is unclear.
Section \ref{sec:lower-bound-limit-set} contains the proof of \autoref{thm:rational-approximation}, section \ref{sec:upper-bound-limit-set} contains the proof of Theorem \ref{cor:geolimset}, and section \ref{sec:fail-quasi-conv} contains the proof of \autoref{thm:qc-fail}.
Sections \ref{sec:lower-bound-limit-set}, \ref{sec:upper-bound-limit-set}, and \ref{sec:fail-quasi-conv} are independent of each other, and can be read in any order.

\subsection*{Acknowledgments}
The author would like to thank Alex Wright for introducing him to the problem, and also Jon Chaika, Christopher Zhang, and Bradley Zykoski, for several helpful conversations throughout the course of the project.
The author also would like to thank the creators of \texttt{surface-dynamics} \cite{vincent_delecroix_2021_5057590}, which helped with many of the experiments that guided the results in this paper.

\section{Background}
\label{sec:background}

\subsection{Non-orientable surfaces and measured foliations}
\label{sec:backgr-meas-foli}

For the purposes of this paper, the most convenient way to think about non-orientable surfaces will be to attach \emph{crosscaps} to orientable surfaces.
Given a surface $S$, attaching a crosscap is the operation of deleting the interior of a small embedded disc, and gluing the boundary $S^1$ via the antipodal map.
Attaching $k$ crosscaps to a genus $g$ surface results in a genus $2g+k$ non-orientable surface $\no_{2g+k}$ (i.e. the non-orientable surface obtained by taking the connect sum of $2g+k$ copies of $\mathbb{RP}^2$).
Associated to each cross cap is a one-sided curve, which is the image of the boundary under the quotient map.
We say that a curve intersects the crosscap if it intersects the associated one-sided curve.
% In particular, any non-orientable surface can be obtained by just attaching $1$ or $2$ crosscaps to an orientable surface.

Consider the set $\scc$ of simple closed curves on a non-orientable surface $\no$.
The elements of $\scc$ can be classified into two types.
\begin{description}
\item[Two sided curves] Tubular neighbourhoods are cylinders.
\item[One sided curves] Tubular neighbourhoods are Möbius bands.
\end{description}
The subset of two sided curves in denoted by $\scc^+$ and one sided curves by $\scc^-$.
Since these two types are topologically distinct, they form invariant subspaces with respect to the mapping class group action.
If we think of our non-orientable surface as an orientable subsurface with crosscaps attached, a two-sided curve is one that intersects an even number of crosscaps, and a one-sided curve is one that intersects an odd number of crosscaps.

The orientable double cover of $\no_g$ is the orientable surface $\os_{g-1}$, and comes with an orientation reversing involution $\iota$.
Since this is an orientation double cover, the subgroup of $\pi_1(\no_g)$ corresponding to this cover is characteristic, i.e. left invariant by every homeomorphism induced automorphism of the fundamental group.
A useful consequence of this fact is that one can lift mapping classes uniquely.
\begin{fact}
  Any self homeomorphism of $\no_g$ lifts to a unique orientation preserving self homeomorphism of $\os_{g-1}$, and as a consequence, one has the injective homomorphism induced by the covering map $p$.
  \begin{align*}
    p^{\ast}: \mcg(\no_d) \hookrightarrow \mcg^+(\os_{d-1})
  \end{align*}
  Furthermore, this inclusion preserves the mapping class type, i.e. finite order, reducible and pseudo-Anosov maps in $\mcg(\no_g)$ stay finite order, reducible, and pseudo-Anosov in $\mcg(\os_{g-1})$.
\end{fact}

One also obtains a map from $\teich(\no_g)$ to $\teich(\os_{g-1})$ using the fact that mapping classes can be lifted canonically.
Given a point $(p, \varphi)$ in $\teich(\no_g)$, where $p$ is a hyperbolic surface homeomorphic to $\no_g$, and $\varphi$ is an isotopy class of homeomorphism from $\no_g$ to $p$, we define the image of $(p, \varphi)$ in $\teich(\os_{g-1})$ to be $(\widetilde{p}, \widetilde{\varphi})$, where $\widetilde{p}$ is the orientation double cover of $p$, and $\widetilde{\varphi}$ is the orientation preserving lift of the homeomorphism $\varphi$.
One can also explicitly describe the image of this map.
To do so, we consider the extended Teichmüller space of $\os_{g-1}$, i.e. also allowing orientation reversing markings.
This space has two connected components, one for each orientation, and there is a canonical involution, given by reversing the orientation, that exchanges the two connected components.
We denote this conjugation map by $\overline{\cdot}$.
There is another involution, induced by the orientation reversing deck transformation of $\os_{g-1}$, which we denote by $\iota^{\ast}$.
This map also exchanges the two components of the extended Teichmüller space.
The image of $\teich(\no_g)$ is precisely the set of points fixed by the composition of these two maps, i.e. $\overline{\iota^{\ast}}$.
We skip the proof of these two facts, since they follow by relatively elementary covering space arguments, and summarize the result in the following theorem.

\begin{theorem}[Embedding Teichm\"uller spaces]
  \label{thm:embedding-teich}
  Given a point $(p, \varphi)$ in $\teich(\no_g)$, there is a unique point $(\widetilde{p}, \widetilde{\varphi})$ in $\teich(\os_{g-1})$, where $\widetilde{p}$ is the pullback of the metric, and $\widetilde{\varphi}$ is the unique orientation preserving lift of the marking. The image of the inclusion map is the intersection of the invariant set of $\overline{\iota^\ast}$ with the connected component of the extended Teichmüller space corresponding to orientation preserving maps.
\end{theorem}

It turns out that the image of $\teich(\no_g)$ in $\teich(\os_{g-1})$ is an isometrically embedded submanifold, and the geodesic flow can be represented by the action of the diagonal subgroup of $SL(2, \mathbb{R})$.

To understand the Teichm\"uller geodesic flow on $\teich(\no_g)$, we need to determine what the cotangent vectors look like: let $X$ be a point in $\teich(\no_{g})$ and let $\wt{X}$ be the corresponding point in $\teich(\os_{g-1})$.
Then the map on the extended Teichmüller space induced by the orientation reversing deck transformation maps $\wt{X}$ to $\overline{\wt{X}}$, i.e. the conjugate Riemann surface. Following that with the canonical conjugation map brings us back to $\wt{X}$.
Let $q$ be a cotangent vector at $\overline{\wt{X}}$, i.e. an anti-holomorphic quadratic differential on the Riemann surface $\overline{\wt{X}}$.
Pulling back $q$ along the canonical conjugation map gives a holomorphic quadratic differential on $X$.
In local coordinate chart on $\wt{X}$, this looks like $q(z) dz^2$ if on the corresponding chart on $\overline{\wt{X}}$ it looked like $q(\overline{z}) dz^2$.
We want this to equal $\iota^{\ast}q$, which will also be a holomorphic quadratic differential on $\wt{X}$.
If that happens, then $\iota^{\ast} q$ is a cotangent vector to the point $X$ in $\teich(\no_g)$.

\begin{example}[A cotangent vector to a point in $\teich(\no_3)$]
  Consider the quadratic differential $q$ on a genus two Riemann surface pictured in \autoref{fig:dqd-example}.
  \begin{figure}[h]
    \centering
    \def\svgscale{0.45}
    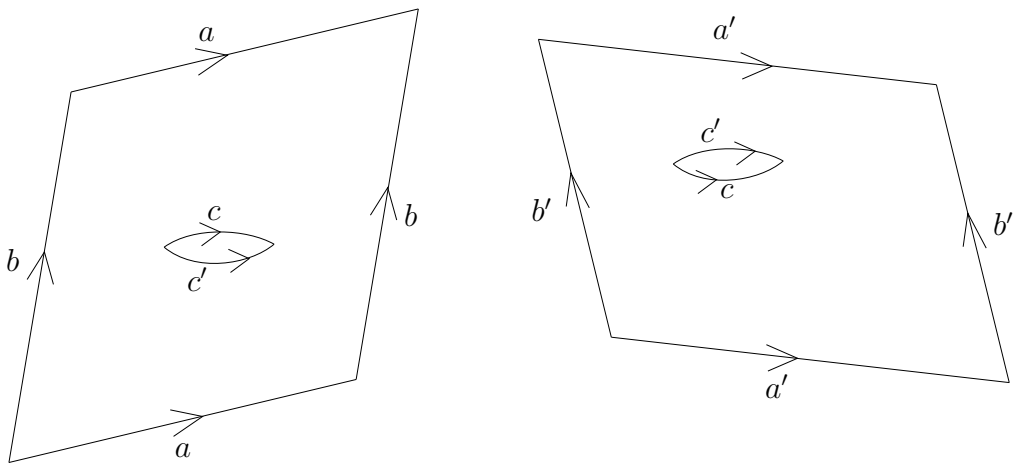

    \caption{A quadratic differential $q$ on $\os_2$ given by the slit torus construction.}
    \label{fig:dqd-example}
  \end{figure}

  Observe that this particular quadratic differential is the global square of an abelian differential, so it makes sense to talk about the pairing between $\sqrt{q}$ and the homology classes $\{a, a^{\prime}, b, b^{\prime}, c, c^{\prime}\}$.
  Recall that the action of a mapping class like $\iota$ is merely relabelling homology classes: in this case $\iota$ swaps $a$ with $-a^{\prime}$, $b$ with $b^{\prime}$, and $c$ with $-c^{\prime}$.
  That gives us the following expressions involving $\sqrt{q}$.
  \begin{align}
    \label{eq:back-12}
    \langle \iota^{\ast}\sqrt{q}, a \rangle &= \langle \sqrt{q}, -a^{\prime} \rangle \\
    \langle \iota^{\ast}\sqrt{q}, b \rangle &= \langle \sqrt{q}, b^{\prime} \rangle \\
    \langle \iota^{\ast}\sqrt{q}, c \rangle &= \langle \sqrt{q}, -c^{\prime} \rangle
  \end{align}
  On the other hand, the conjugation action conjugates the complex value of each pairing.
  \begin{align}
    \label{eq:back-13}
    \langle \overline{\sqrt{q}}, a \rangle &= \overline{  \langle \sqrt{q}, a \rangle } \\
    \langle \overline{\sqrt{q}}, b \rangle &= \overline{  \langle \sqrt{q}, b \rangle } \\
    \langle \overline{\sqrt{q}}, c \rangle &= \overline{  \langle \sqrt{q}, c \rangle }
  \end{align}
  For $q$ to be invariant under $\overline{\iota}$, both of the above set of equations must be satisfied, which imposes certain conditions on $q$.
  For instance, if the complex lengths of $a$ and $a^{\prime}$ to be conjugates of each other, the complex lengths of $b$ and $b^{\prime}$ to be negative conjugates of each other, and forces the complex length of $c$ and $c'$ to be real.
  Only the quadratic differentials satisfying these constraints will be the cotangent vectors to points in the image of $\teich(\no_3)$.

  To realize the quadratic differential directly as an object on $\no_3$, we can quotient out the flat surface given by $q$ by the orientation reversing deck transformation.
  Doing that for our example gives the non-orientable flat surface gives the picture seen in \autoref{fig:dqd-no3}.
  \begin{figure}[h]
    \centering
    \def\svgscale{0.45}
    %% Creator: Inkscape 1.0.1 (3bc2e813f5, 2020-09-07), www.inkscape.org
%% PDF/EPS/PS + LaTeX output extension by Johan Engelen, 2010
%% Accompanies image file 'dqd-no3.pdf' (pdf, eps, ps)
%%
%% To include the image in your LaTeX document, write
%%   \input{<filename>.pdf_tex}
%%  instead of
%%   \includegraphics{<filename>.pdf}
%% To scale the image, write
%%   \def\svgwidth{<desired width>}
%%   \input{<filename>.pdf_tex}
%%  instead of
%%   \includegraphics[width=<desired width>]{<filename>.pdf}
%%
%% Images with a different path to the parent latex file can
%% be accessed with the `import' package (which may need to be
%% installed) using
%%   \usepackage{import}
%% in the preamble, and then including the image with
%%   \import{<path to file>}{<filename>.pdf_tex}
%% Alternatively, one can specify
%%   \graphicspath{{<path to file>/}}
%% 
%% For more information, please see info/svg-inkscape on CTAN:
%%   http://tug.ctan.org/tex-archive/info/svg-inkscape
%%
\begingroup%
  \makeatletter%
  \providecommand\color[2][]{%
    \errmessage{(Inkscape) Color is used for the text in Inkscape, but the package 'color.sty' is not loaded}%
    \renewcommand\color[2][]{}%
  }%
  \providecommand\transparent[1]{%
    \errmessage{(Inkscape) Transparency is used (non-zero) for the text in Inkscape, but the package 'transparent.sty' is not loaded}%
    \renewcommand\transparent[1]{}%
  }%
  \providecommand\rotatebox[2]{#2}%
  \newcommand*\fsize{\dimexpr\f@size pt\relax}%
  \newcommand*\lineheight[1]{\fontsize{\fsize}{#1\fsize}\selectfont}%
  \ifx\svgwidth\undefined%
    \setlength{\unitlength}{373.35392509bp}%
    \ifx\svgscale\undefined%
      \relax%
    \else%
      \setlength{\unitlength}{\unitlength * \real{\svgscale}}%
    \fi%
  \else%
    \setlength{\unitlength}{\svgwidth}%
  \fi%
  \global\let\svgwidth\undefined%
  \global\let\svgscale\undefined%
  \makeatother%
  \begin{picture}(1,1.02005813)%
    \lineheight{1}%
    \setlength\tabcolsep{0pt}%
    \put(0,0){\includegraphics[width=\unitlength,page=1]{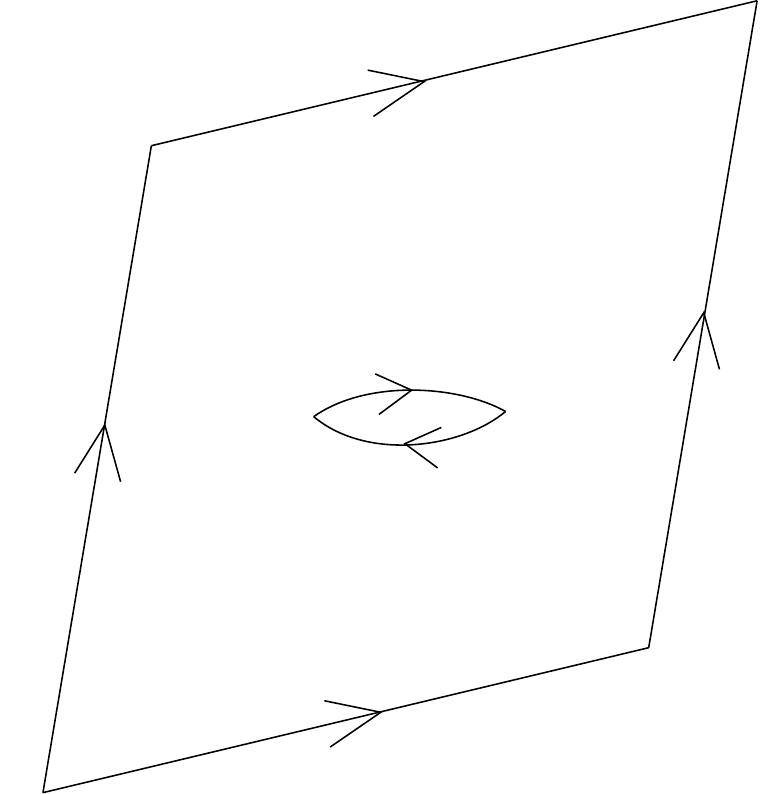}}%
    \put(0.4805943,0.94681568){\color[rgb]{0,0,0}\makebox(0,0)[lt]{\lineheight{1.25}\smash{\begin{tabular}[t]{l}$a$\end{tabular}}}}%
    \put(0.42701759,0.01259884){\color[rgb]{0,0,0}\makebox(0,0)[lt]{\lineheight{1.25}\smash{\begin{tabular}[t]{l}$a$\end{tabular}}}}%
    \put(0.03542927,0.44806521){\color[rgb]{0,0,0}\makebox(0,0)[lt]{\lineheight{1.25}\smash{\begin{tabular}[t]{l}$b$\end{tabular}}}}%
    \put(0.94136765,0.53083511){\color[rgb]{0,0,0}\makebox(0,0)[lt]{\lineheight{1.25}\smash{\begin{tabular}[t]{l}$b$\end{tabular}}}}%
    \put(0.49974254,0.54477871){\color[rgb]{0,0,0}\makebox(0,0)[lt]{\lineheight{1.25}\smash{\begin{tabular}[t]{l}$c$\end{tabular}}}}%
    \put(0.46496328,0.38824589){\color[rgb]{0,0,0}\makebox(0,0)[lt]{\lineheight{1.25}\smash{\begin{tabular}[t]{l}$c$\end{tabular}}}}%
  \end{picture}%
\endgroup%

    \caption{A quadratic differential on $\no_3$.}
    \label{fig:dqd-no3}
  \end{figure}
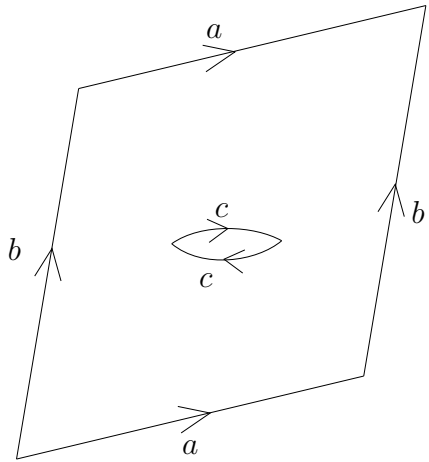
\end{example}
This example suggests what the right definition of a quadratic differential on a non-orientable surface ought to be: in the flat picture, rather than allowing gluing via just the maps $z \mapsto \pm z + c$, we also allow $z \mapsto \pm \overline{z} + c$.
This leads to the definition of \emph{dianalytic quadratic differentials} (which we'll abbreviate to DQDs).

\begin{definition}[Dianalytic quadratic differential (adapted from \autocite{Wright2015})]
  A dianalytic quadratic differential is the quotient of a collection of polygons in $\mathbb{C}$, modulo
  certain equivalences.  The quotienting satisfies the following conditions.
  \begin{enumerate}[(1)]
  \item The interiors of the polygons are disjoint.
  \item Each edge is identified with exactly one other edge, and the mapping must be of one of the following
    four forms: $z \mapsto z + c$, $z \mapsto -z + c$, $z \mapsto \overline{z} + c$, or
    $z \mapsto -\overline{z} + c$.
  \item Extending the edge identification map to a small enough open neighbourhood of a point on the edge
    should not map it to an open neighbourhood of the image of the point: in other words, it should get mapped
    to the ``other side'' of the edge.
  \end{enumerate}
  Two such quotiented collections of polygons are considered the same if they differ by a composition of the following
  moves.
  \begin{enumerate}[(1)]
  \item A polygon may be translated, rotated by $\pi$ radians, or reflected across the real or imaginary axis.
  \item A polygon may be cut along a straight line to form two polygons, or two polygons sharing an edge may
    be glued together to form a single polygon.
  \end{enumerate}
\end{definition}

Given a DQD, we can pull it back to the orientation double cover, getting an actual quadratic differential: this operation corresponds to identifying a cotangent vector to a point in $\teich(\no_g)$ to the corresponding cotangent vector in $\teich(\os_{g-1})$.

To verify that $\teich(\no_g)$ is isometrically embedded, all we need to do is verify that the Teichm\"uller geodesic flow takes the quadratic differentials satisfying the symmetry condition $\iota^{\ast}(q) = \overline{q}$ to quadratic differentials that satisfy the symmetry conditions.

\begin{lemma}
  \label{lem:gt-invariance}
  If $q$ satisfies $\iota^{\ast}(q) = \overline{q}$, then for any $t$, $\iota^{\ast}(g_tq) = \overline{g_tq}$.
\end{lemma}
\begin{proof}
  Recall that if $q$ satisfies the given condition, we must have the following hold for any homology class $a$.
  \begin{equation}
    \label{eq:back-14}
    \langle \sqrt{q}, \iota(a) \rangle = \overline{ \langle \sqrt{q}, a \rangle}
  \end{equation}
  If $q$ is not the global square of an abelian differential, we may have to pass to the holonomy double
  cover. Observe now what $g_t$ does to $q$.
  \begin{equation}
    \label{eq:back-15}
    \langle \sqrt{g_t q}, \iota(a) \rangle = e^t \mathrm{Re} \langle \sqrt{q}, \iota(a) \rangle + i e^{-t} \mathrm{Im}\langle \sqrt{q}, \iota(a) \rangle
  \end{equation}
  Using \eqref{eq:back-14}, we simplify \eqref{eq:back-15} to the following.
  \begin{align}
    \label{eq:back-16}
    \langle \sqrt{g_t q}, \iota(a) \rangle &= e^t \mathrm{Re} \langle \sqrt{q}, a \rangle - i e^{-t} \mathrm{Im}\langle \sqrt{q}, a \rangle \\
                                           &= \overline{\langle \sqrt{g_tq}, a \rangle}
  \end{align}
  This proves the lemma.
\end{proof}
\begin{remark}
  The key idea that diagonal matrices commute:
  the conjugation action is really multiplication by $
  \begin{pmatrix}
    1 & 0 \\
    0 & -1
  \end{pmatrix}
  $ which happens to commute with the diagonal matrices of determinant $1$, which are exactly the matrices
  corresponding to geodesic flow. On the other hand, the conjugation matrix does not commute with the horocycle
  flow matrices, and that shows that the horocycle flow is not well defined on the cotangent bundle
  of $\teich(\no_g)$.
\end{remark}
Lemma \ref{lem:gt-invariance} shows that the Teichm\"uller geodesic flow for the cotangent bundle of $\teich(\no_g)$ is the restriction of the geodesic flow for the ambient space $\teich(\os_{g-1})$.

\autoref{thm:embedding-teich} gives us an alternative perspective into the action of $\mcg(\no_g)$ on $\teich(\no_g)$.
$\mcg(\no_g)$ can be thought of as the subgroup of $\mcg(\os_{g-1})$ that stabilizes a totally real isometrically embedded submanifold $\teich(\no_g)$.
With this perspective, $\mcg(\no_g)$ can be thought of as the higher dimensional generalization of the subgroups obtained by stabilizing Teichmüller discs, i.e. Veech groups.

We now state a few classical results about measured foliations on non-orientable surfaces that show why the theory diverges significantly from the orientable case.

A measured foliation on a non-orientable surface $\no_g$ is singular foliation along with an associated transverse measure, up to equivalence by Whitehead moves.
Any leaf of a measured foliation can either be non-compact or compact: in the former case, the closure of the non-compact leaf fills out a subsurface.
Restricted to the subsurface given by the closure of a non-compact leaf, the foliation is minimal, i.e. the orbit of every point under the flow given by the foliation is dense.
For a compact leaf, there are two possibilities for the topology of the subsurface containing it: if the closed leaf is the core curve or the boundary curve of an embedded M\"obius strip, then the subsurface is the maximal neighbourhood of the periodic leaf that is foliated by periodic leaves as well, and this turns out to be an embedded M\"obius strip.
If the compact leaf is not the core curve or the boundary curve of an embedded M\"obius strip, then it is the core curve of an embedded cylinder, and the maximal neighbourhood of the periodic leaf foliated by periodic leaf is an embedded cylinder.
The identification of leaves with associated subsurfaces lets us decompose a measured foliation into its minimal components.
Note the slightly confusing terminology: when the minimal component is a M\"obius strip or a cylinder, then the foliation restricted to the component is not minimal, but when the minimal component has higher genus, then the foliation restricted to that component indeed is minimal.

We denote the set of measured foliations on $\no_g$ by $\mf(\no_g)$, the set of foliations whose minimal components do not contain a M\"obius strip by $\mf^+(\no_g)$, and the set of foliations whose minimal components contain at least one M\"obius strip by $\mf^-(\no_g)$.
Via the standard identification between simple closed curves and measured foliations, we can associate $\mathbb{Q}$-weighted two-sided multicurves on $\no_g$ to a subset of $\mf^+(\no_g)$, denoted by $\mf^+(\no_g, \mathbb{Q})$.

Quotienting out $\mf(\no_g)$ by the $\mathbb{R_+}$-action given by scaling the transverse measure gives us the set of projective measured foliations $\pmf(\no_g)$.
The subsets $\mf^-(\no_g)$, $\mf^+(\no_g)$, and $\mf^+(\no_g, \QQ)$ are $\mathbb{R}$-invariant, and thus descend to their projective versions $\pmf^-(\no_g)$, $\pmf^+(\no_g)$, and $\pmf^+(\no_g, \QQ)$.
The set $\pmf(\no_g)$ is the boundary of the Teichm\"uller space of $\no_g$, and admits a continuous mapping class group action.
It is when considering the mapping class group action that we see differences between the orientable and the non-orientable case.

\begin{theorem}[Proposition 8.9 of \cite{gendulphe_whats_2017}]
  \label{thm:full-non-minimal}
  The action of $\mcg(\no_g)$ (for $g \geq 2$) on $\pmf(\no_g)$ is not minimal.
  In fact, the action is not even topologically transitive.
\end{theorem}

Compare this to the case of $\mcg(\os_g)$.
\begin{theorem}[Theorem 6.19 of \cite{fathi2012thurston}]
  \label{thm:orientable-orbit-closure}
  The action of $\mcg(\os_g)$ on $\pmf(\os_g)$ is minimal.
\end{theorem}

\begin{remark}
  The proof of non minimality and topological non-transitivity in the non-orientable case follow from the fact that one can construct a $\mcg(\no_g)$-invariant continuous function on $\mf(\no_g)$.
  That is because starting with a foliation in $\mf^+(\no_g)$, it is impossible to approximate an element of $\mf^-(\no_g)$ since one does not have Dehn twists about one-sided curves.
\end{remark}

One can now consider subspaces of $\mf(\no_g)$ where the $\mcg(\no_g)$ action might be nicer. There are two natural subspaces: $\mf^+(\no_g)$, and $\mf^{-}(\no_g)$.
Danthony-Nogueira proved the following theorem about $\mf^{-}(\no_g)$ in \cite{ASENS_1990_4_23_3_469_0}.

\begin{theorem}[Theorem II of \cite{ASENS_1990_4_23_3_469_0}]
  \label{thm:DN90-2}
  $\mf^{-}(\no_g)$ is an open dense subset of $\mf(\no_g)$ of full Thurston measure.
\end{theorem}

\autoref{thm:DN90-2} means that the $\mcg(\no_g)$-orbit closure in $\pmf(\no_g)$ of any point in $\teich(\no_g)$ is contained in $\pmf^+(\no_g)$.
In the case of $\mcg(\os_g)$, $\pmf^+(\os_g) = \pmf(\os_g)$, and the orbit closure is actually all of $\pmf(\os_g)$.
\begin{corollary}[Corollary of \autoref{thm:orientable-orbit-closure}]
  For any $x \in \teich(\os_g)$, $\overline{\mcg(\os_g) \cdot x} \cap \pmf(\os_g) = \pmf(\os_g)$.
\end{corollary}

\autoref{thm:full-non-minimal} and \autoref{thm:DN90-2} suggest that studying the $\mcg(\no_g)$ dynamics restricted to $\mf^{-}(\no_g)$ will be hard since one will not have minimality, or ergodicity with respect to any measure with full support.
In Section \ref{sec:lower-bound-limit-set}, we get a lower bound for the set on which $\mcg(\no_g)$ acts minimally.

\subsection{Limit sets of mapping class subgroups}
\label{sec:backgr-limit-sets}

The first results on limit sets of subgroups of mapping class groups were obtained by Masur for handlebody subgroups \cite{masur_1986}, and McCarthy-Papadopoulos for general mapping class subgroups \cite{McCarthy1989}.
They defined two distinct notions of limit sets; while they did not give distinct names to the two different definitions, we will do so for the sake of clarity.
\begin{definition}[Dynamical limit set]
  Given a subgroup $\Gamma$ of the mapping class group, the dynamical limit set $\dynlim(\Gamma)$ is the minimal closed invariant subset of $\pmf$ under the action of $\Gamma$.
\end{definition}
\begin{definition}[Geometric limit set]
  Given a subgroup $\Gamma$ of the mapping class group, and a point $x$ in the Teichmüller space, its boundary orbit closure $\Lambda_{\mathrm{geo}, x}(\Gamma)$ is intersection of its orbit closure with the Thurston boundary, i.e. $\overline{\Gamma x} \cap \pmf$.
  The geometric limit set is the union of all boundary orbit closures, as we vary $x$ in the Teichmüller space, i.e. $\geolim(\Gamma) = \bigcup_{x \in \teich} \Lambda_{\mathrm{geo}, x}(\Gamma)$.
\end{definition}

\begin{remark}
  The specific family of subgroups considered by McCarthy-Papadopoulos were subgroups containing at least two non-commuting pseudo-Anosov mapping classes, in which case the dynamical limit set is unique.
  The mapping class groups $\mcg(\no_{g})$ considered as a subgroup of $\mcg(\os_{g-1})$ certainly satisfies this property, letting us talk about \emph{the} dynamical limit set.
\end{remark}

Both of these definitions are natural generalizations of the limit sets of Fuchsian groups acting on $\HH^2$.
In the hyperbolic setting, the two notions coincide, but for mapping class subgroups, the dynamical limit set may be a proper subset of the geometric limit set.

For simple enough subgroups, one can explicitly work out $\dynlim(\Gamma)$ and $\geolim(\Gamma)$: for instance, when $\Gamma$ is the stabilizer of the Teichmüller disc associated to a Veech surface, $\dynlim(\Gamma)$ is the visual boundary of the Teichmüller disc, which by Veech dichotomy, only consists of either uniquely ergodic directions on the Veech surface, or the cylinder directions, where the coefficients on the cylinders are their moduli in the surface.
On the other hand, $\geolim(\Gamma)$ consists of all the points in $\dynlim(\Gamma)$, but it additionally contains all possible convex combinations of the cylinders appearing in $\dynlim(\Gamma)$ (see Section 2.1 of \cite{2007math......2034K}).

The gap between $\geolim$ and $\dynlim$ suggests the following operation on subsets of $\pmf$, which we will call \emph{saturation}.
\begin{definition}[Saturation]
  Given a projective measured foliation $\lambda$, we define its saturation $\expansion(\lambda)$ to be the image in $\pmf$ of set of all non-zero measures invariant measures on the topological foliation associated to $\lambda$.
  Given a subset $\Lambda$, we define its saturation $\expansion(\Lambda)$ to be the union of saturations of the projective measured laminations contained in $\Lambda$.
\end{definition}
Observe that for a uniquely ergodic foliation $\lambda$, $\expansion(\lambda) = \{\lambda\}$, for a minimal but not uniquely ergodic $\lambda$, $\expansion(\lambda)$ is the convex hull of all the ergodic measures supported on the topological lamination associated to $\lambda$, and for a foliation with all periodic leaves, $\expansion(\lambda)$ consists of all foliations that can be obtained by assigning various weights to the core curves of the cylinders.

Going back to the example of the stabilizer of the Teichmüller disc of a Veech surface, we see that $\geolim(\Gamma) = \expansion(\dynlim(\Gamma))$.
One may ask if this is always the case.
\begin{question}
  Is $\geolim(\Gamma) = \expansion(\dynlim(\Gamma))$ for all $\Gamma$?
\end{question}
We know from \autoref{thm:rational-approximation} that $\geolim(\Gamma)$ is contained in $\expansion(\dynlim(\Gamma))$ when $\Gamma = \mcg(\no_g)$.

McCarthy-Papadopoulos also formulated an equivalent definition of $\dynlim(\Gamma)$, which is easier to work with in practice.

\begin{untheorem}[Theorem 4.1 of \cite{McCarthy1989}]
  \label{thm:equivalence-of-limit-sets}
  $\dynlim(\Gamma)$ is the closure in $\pmf$ of the stable and unstable foliations of all the pseudo-Anosov mapping classes in $\Gamma$.
\end{untheorem}

\subsection*{List of notation}
Here we describe some of the more commonly used symbols in the paper.
\begin{itemize}
\item[] $\os_g$: The compact orientable surface of genus $g$.
\item[] $\no_{g}$: The compact non-orientable surface of genus $g$.
\item[] $\iota$: The deck transformation of the orientation double cover of a non-orientable surface.
\item[] $\teich(S)$: The Teichm\"uller space of $S$.
\item[] $\systole(\no_d)$: The set of points in $\teich(\no_d)$ where no one-sided curve is shorter than
  $\varepsilon$.
% \item[] $g_t$: The Teichm\"uller geodesic flow on quadratic differentials.
\item[] $\mcg(S)$: The mapping class group of $S$.
% \item[] $\scc(S)$: The set of all simple closed curves on $S$.
\item[] $\mf(S)$: The space of measured foliations on $S$.
\item[] $\pmf(S)$: The space of projective measured foliations on $S$.
\item[] $\mf^+(\no_d)$, $\pmf^+(\no_d)$: The set of (projective) measured foliations on $\no_d$ containing
  no one-sided leaves.
\item[] $\mf^-(\no_d)$, $\pmf^-(\no_d)$: The set of (projective) measured foliations on $\no_d$ containing
  some one-sided leaf.
\item[] $\mf(S; \QQ)$, $\pmf(S; \QQ)$: The set of all (projective) weighted rational multicurves on $S$.
\item[] $\geolim(\Lambda)$: The geometric limit set of the discrete group $\Lambda$.
\item[] $\dynlim(\Lambda)$: The dynamical limit set of the discrete group $\Lambda$.
\item[] $\ell_i(\gamma)$: The hyperbolic length of $\gamma$ on the surface $m_i$, where $\left\{ m_i \right\}$ is a sequence in the Teichmüller space. We use this when we are only talking about hyperbolic lengths. When talking about both hyperbolic and flat lengths, we disambiguate them using the following symbols.
\item[] $\lhyp(M, \gamma)$: The hyperbolic length of $\gamma$ with respect to the hyperbolic structure on $M \in \teich(S)$. We will suppress $M$ when it is clear from context.
\item[] $\lflat(q, \gamma)$: The flat length of $\gamma$ with respect to the flat structure given by the DQD $q$. We will suppress $q$ when it is clear from context.
\item[] $\mu_{c}$: The probability measure on a transverse arc given by the closed curve $c$.
\end{itemize}

\section{Lower bound for the limit set}
\label{sec:lower-bound-limit-set}

A natural lower bound for $\dynlim(\no_g)$ is the closure of the set of rational two-sided multicurves $\pmf^+(\no_g, \QQ)$.
For any $\lambda \in \pmf^+(\no_g, \QQ)$, and any psuedo-Anosov $\gamma$, conjugating $\gamma$ with large enough powers of the Dehn multi-twist given by $\lambda$ gives us a sequence of pseudo-Anosov maps whose stable foliation approaches $\lambda$, which shows that $\dynlim(\no_g)$ must contain $\lambda$.
Note that the same argument does not work if $\lambda \in \pmf^-(\no_g, \QQ)$, since one cannot Dehn twist about one-sided curves.
In Section \ref{sec:upper-bound-limit-set}, we show that the geometric limit set is indeed contained in the complement of $\pmf^-(\no_g)$.

In \cite{gendulphe_whats_2017}, Gendulphe made the following conjecture about $\overline{\pmf^+(\no_g, \QQ)}$.
\begin{conjecture}[Conjecture 9.1 of \cite{gendulphe_whats_2017}]
  \label{conj:gendulphe-1}
  For $g \geq 4$, $\pmf^+(\no_g) = \overline{\pmf^+(\no_g, \QQ)}$.
\end{conjecture}
We prove a slightly weaker version of the above conjecture, by describing a subset of the foliations that can be approximated by multicurves in $\pmf^+(\no_g, \QQ)$.
To state the theorem, we need to define what it means for a minimal foliation to be orientable.
\begin{definition}[Orientable foliation]
  A local orientation on a foliation is the choice of a locally constant tangent direction on the leaves in a small open set.
  If the local orientation can be extended to an entire minimal foliation, the foliation is said to be orientable.
\end{definition}
In the setting of orientable surfaces, the vertical foliations of translation surfaces are orientable, while there are some directions in half-translation surfaces where the foliation is non-orientable.
There exist similar examples of orientable and non-orientable foliations on non-orientable surfaces.

Having defined the notion of orientable foliations, we can state the main theorem of this section.
\begin{theorem}
  \label{thm:rational-approximation}
  A foliation $\lambda \in \pmf^+(\no_g)$ can be approximated by foliations in $\pmf^+(\no_g, \QQ)$ if all the minimal components $\lambda_j$ of $\lambda$ satisfy one of the following criteria.
  \begin{enumerate}[(i)]
  \item $\lambda_j$ is periodic.
  \item $\lambda_j$ is ergodic and orientable.
  \item $\lambda_j$ is uniquely ergodic.
  \end{enumerate}
  Furthermore, if $\lambda_j$ is minimal, but not uniquely ergodic, there exists some other foliation $\lambda_j^{\prime}$ supported on the same topological foliation as $\lambda_j$ that can be approximated by elements of $\pmf^+(\no_g, \QQ)$.
\end{theorem}

Before we prove this result, we need to define the \emph{orbit measure} associated to simple curve, and define what it means for an orbit measure to be \emph{almost invariant}.
Consider an arc $\eta$ transverse to a measured foliation $\lambda$.
We assign one of the sides of $\eta$ to be the ``up'' direction, and the other side to be the ``down'' direction.
This lets us define the first return map to $T$.
\begin{definition}[First return map]
  The first return map $T$ maps a point $p \in \eta$ to the point obtained by flowing along the foliation in the ``up'' direction until the flow intersects $\eta$ again.
  The point of intersection is defined to be $T(p)$.
  If the flow terminates at a singularity, $T(p)$ is left undefined: there are only countable many points in $\eta$ such that this happens.
\end{definition}
Since $\lambda$ is a measured foliation, it defines a measure on $\eta$: we can scale it so that it is a probability measure.
It follows from the definition of transverse measures that the measure is $T$-invariant.
It is a classical result of Katok \cite{zbMATH03467479}  and Veech \cite{Veech1978} that the set of $T$-invariant probability measures is a finite dimensional simplex contained in the Banach space of bounded signed measures on $\eta$.
Given an orbit of a point $p$ under the $T$-action of length $L$, we construct a probability measure on $\eta$, called the orbit measure of $p$.
\begin{definition}[Orbit measure]
  The orbit measure of length $L$ associated to the point $p$ is the following probability measure on $\eta$.
  \begin{align*}
    \mu_{p, L} \coloneqq \frac{1}{L} \sum_{i=0}^{L-1} \delta_{T^i(p)}
  \end{align*}
  Here, $\delta_{x}$ is the Dirac delta measure at the point $x$.
\end{definition}
One might expect that if a point $p$ equidistributes, then a long orbit measure starting at $p$ will be ``close'' to an invariant measure.
We formalize this notion by metrizing the Banach space of signed finite measures on $\eta$.
\begin{definition}[Lèvy-Prokhorov metric]
  Define $\norm{\cdot}_{\mathrm{BL}}$ denote the bounded Lipschitz norm on the space of continuous functions on $\eta$.
  \begin{align*}
    \norm{f}_{\mathrm{BL}} \coloneqq \norm{f}_{\infty} + \sup_{x \neq y} \frac{\left| f(x) - f(y) \right|}{\left| x - y \right|}
  \end{align*}
\end{definition}
Then the Lèvy-Prokhorov distance $d_{\mathrm{LP}}$ between the probability measures $\mu_1$ and $\mu_2$ is defined to be the following.
\begin{align*}
  d_{\mathrm{LP}}(\mu_1, \mu_2) \coloneqq \sup_{\norm{f}_{\mathrm{BL}} \leq 1} \int f (\dd \mu_1 - \dd \mu_2)
\end{align*}
Using the Lèvy-Prokhorov metric, we can define what it means for a probability measure to be $\varepsilon$-almost $T$-invariant.
\begin{definition}[$\varepsilon$-almost $T$-invariance]
  A measure $\mu$ is $\varepsilon$-almost $T$-invariant if $d_{\mathrm{LP}}(\mu, T \mu) \leq \varepsilon$. Here $T \mu$ is the pushforward of $\mu$ under $T$.
\end{definition}
We state the following easy fact about orbit measures without proof.
\begin{fact}
  An orbit measure of length $L$ is $\frac{2}{L}$-almost $T$-invariant.
\end{fact}
The following lemma shows that a long orbit measure is close to an invariant measure.
\begin{lemma}
  \label{lem:long-orbit-is-almost-invariant}
  Let $\left\{ n_j \right\}$ be a sequence of positive integers and let $\{\mu_{ij}\}$ be orbit measures such that $1 \leq i \leq n_j$ and $d(\mu_{ij}, T{\mu}_{ij}) \leq l_j$, where $\lim_{j \to \infty} l_j = 0$.
  For any $\varepsilon > 0$, there exists an $J$ large enough such that for all $j > J$, $\mu_{ij}$ is within distance $\varepsilon$ of an invariant measure.
\end{lemma}
\begin{proof}
  Let $A_k$ be the closure of all the $\mu_{ij}$ such that $j \geq k$.
  The set $A_k$ is compact, because it is a closed subset of a compact set, and we have that $\bigcap_{k=1}^{\infty} A_k$ is contained in the set of invariant measures.
  By compactness, we have that for some large enough $J$, $A_J$ must be in a $\varepsilon$-neighbourhood of the set of invariant measures, and therefore every $\mu_{ij}$ for $j > J$ must distance at most $\varepsilon$ away from an invariant measure.
\end{proof}
We now sketch a proof of the following lemma about simplices in finite dimensional normed spaces.
\begin{lemma}
  \label{lem:finite-normed}
  Let $V$ be a finite dimensional normed vector space, and $S$ be a simplex in $V$.
  Let $\{p_1, \ldots, p_n\}$ be points in $S$ such that they are all at least distance $\varepsilon$ from a vertex $v$.
  Then there exists a positive constant $k$ such that any convex combination of $\{p_i\}$ is distance at least $\dfrac{\varepsilon}{k}$ from $v$.
\end{lemma}
\begin{proof}[Sketch of proof]
  We shift the simplex so that the vertex $v$ is at the origin.
  It will also suffice to let $\{p_1, \ldots, p_n\}$ be the vectors joining $0$ to the other vertices scaled to have norm $\varepsilon$.
  The convex combinations of $\{p_1, \ldots, p_n\}$ will form a compact set not containing $0$.
  Since the norm is a continuous function, the norm will achieve a minimum $\varepsilon^{\prime}$ on that compact set, and the minimum will not be $0$.
  Then $k = \dfrac{\varepsilon^{\prime}}{\varepsilon}$ is the required value of $k$.
\end{proof}

We now prove a lemma that gives us a criterion for deducing when a long orbit measure is close to an ergodic measure.
\begin{lemma}
  \label{lem:convexity-argument}
  Let $\{n_i\}$ be a sequence of positive integers, and let $\{p_{ij}\}$ and $\{L_{ij}\}$ be points in $\eta$ and positive integers respectively, where $1 \leq j \leq n_i$ and $\min_j L_{ij}$ goes to $\infty$ as $i$ goes to $\infty$.
  Consider the following sequence of probability measures, indexed by $i$.
  \begin{align*}
    \mu_i \coloneqq \frac{\sum_{j=1}^{n_i} L_{ij} \cdot \mu_{p_{ij}, L_{ij}}}{\sum_{j=1}^{n_i} L_{ij}}
  \end{align*}
  If the sequence $\{\mu_i\}$ converges to an ergodic measure $\nu$, then there exists a subsequence of the orbit measures $\mu_{p_{ij}, L_{ij}}$ also converging to $\nu$.
\end{lemma}
\begin{proof}
  Suppose for the sake of a contradiction that no subsequence of $\mu_{p_{ij}, L_{ij}}$ converged to $\nu$.
  That would mean there exists a small enough $\varepsilon > 0$ and a large enough $i_0$ such that for all $i > i_0$, the measures $\mu_{p_{ij}, L_{ij}}$ are more than distance $\varepsilon$ from $\nu$.
  Since $\min_{j} L_{ij}$ goes to $\infty$, there exists some other large enough $i_1 > i_0$ such that for all $i > i_1$, $\mu_{p_{ij}, L_{ij}}$ is within distance $\frac{\varepsilon}{k}$ of the simplex of invariant probability measures, where $k$ is a large integer we will pick later: this is a consequence of \autoref{lem:long-orbit-is-almost-invariant}.
  Using this, we decompose $\mu_{p_{ij}, L_{ij}}$ as the sum of an invariant measure $\iota_{ij}$ and a signed measure $e_{ij}$, such that $d_{\mathrm{LP}}(0, e_{ij}) \leq \frac{\varepsilon}{k}$.
  \begin{align*}
    \mu_{p_{ij}, L_{ij}} = \iota_{ij} + e_{ij}
  \end{align*}
  Observe that the weighted average of $\mu_{p_{ij}, L_{ij}}$ will differ from the weighted average of $\iota_{ij}$ by at most $\frac{\varepsilon}{k}$.
  Also note that all the invariant measures $\iota_{ij}$ are distance at least $\varepsilon - \frac{\varepsilon}{k}$ from $\nu$.
  Since $\nu$ is the vertex of a finite-dimensional convex set, we know from \autoref{lem:finite-normed} that any weighted average of the $\iota_{ij}$ must be at least distance $\frac{\varepsilon - \frac{\varepsilon}{k}}{k^{\prime}}$ from $\nu$, where the multiplicative factor $k^{\prime}$ only depends on the geometry of the convex set of invariant probability measures, and not $\varepsilon$ or $k$.
  By picking $k > 2k^{\prime}$ we can ensure that any weighted average of the $\mu_{p_{ij}, L_{ij}}$ must be at least distance $\frac{\varepsilon}{2k^{\prime}}$ from $\nu$.
  But this would contradict our hypothesis that the measures $\mu_i$ converge to $\nu$.
  That would that there exists some subsequence of $\mu_{p_{ij}, L_{ij}}$ that converges to $\nu$, which proves the lemma.
\end{proof}

We now have everything we need to prove \autoref{thm:rational-approximation}.
\begin{proof}[Proof of \autoref{thm:rational-approximation}]
  If a minimal component $\lambda_j$ is periodic, then the proof is straightforward.
  Since $\lambda$ contains no one-sided component, the core curve of $\lambda_j$ must be two-sided, possibly with an irrational coefficient.
  Approximating the core curve with rational coefficients proves the result in case (i).

  In case (ii), we have that $\lambda_j$ is not periodic, but an ergodic orientable foliation.
  Pick an arc $\eta_0$ transverse to $\lambda_j$ such that the leaf passing through the left endpoint $p_0$ of $\eta_0$ equidistributes with respect to the ergodic transverse measure of $\lambda_j$.
  We can find such a leaf because almost every leaf equidistributes with respect to the ergodic measure.
  We now inductively define a sequence of points $\{p_i\}$, sequence of sub-intervals $\eta_i$, and a sequence of segments $\{a_i\}$ of the leaf passing through $p_0$.
  Let $p_1$ be the first return of the leaf going up through $p_0$ to the interval $\eta_0$.
  Define the sub-interval $\eta_1$ to be the sub-interval whose left endpoint is $p_0$ and right endpoint is $p_1$.
  Let $a_1$ be the segment of the leaf starting at $p_0$ and ending at $p_1$.
  Given a point $p_i$, define $p_{i+1}$ to be the first return to the interval $\eta_{i}$, $\eta_{i+1}$ to be the interval whose left endpoint is $p_0$ and right endpoint is $p_{i+1}$, and $a_{i+1}$ to be the segment of the leaf starting at $p_i$ and ending at $p_{i+1}$.
  % Let $A_{i}$ be the concatenation of the leaf segments $\{a_1, \ldots, a_i\}$.

  Since we have assumed $\lambda_j$ is an orientable foliation, we have that the leaf we are working with always enters $\eta_0$ from the bottom, and exits from the top.
  If we pick $\eta_0$ to be small enough, we can pick a local orientation, and keep track of how a positively oriented frame returns to each $p_i$, i.e. with or without the orientation flipped (see \autoref{fig:two-possibilities}).
  \begin{figure}[h]
    \centering
    \def\svgscale{0.45}
    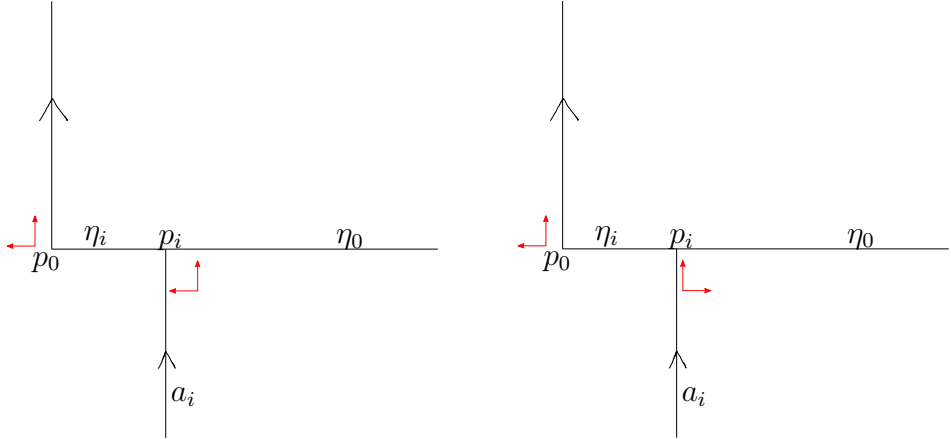

    \caption{Two possibilities for first return to $\eta_i$: on the left, the arc returns without the local orientation flipping, and on the right, the arc returns with the local orientation flipped.}
    \label{fig:two-possibilities}
  \end{figure}
  If the flow returns infinitely often without the orientation flipped, we join the endpoint $p_i$ to $p_0$ by going left along $\eta_i$ to get a simple closed curve that is two-sided.
  Furthermore, the geodesic tightening of the resulting curve is very close to the original curve, because the initial and final tangent vectors can be made arbitrarily close since they both face the ``up'' direction: the Anosov closing lemma then tells us that an orbit of the geodesic flow that approximately closes up can be perturbed by a small amount to exactly close up.
  This gives us a long geodesic that equidistributes with respect to the ergodic measure, and therefore an approximation by two-sided curves.

  If the flow does not return without the orientation flipped infinitely often, it must always return with the orientation flipped after some large enough $i_0$.
  In that case, consider the simple two-sided curves $c_i$ obtained by concatenating $a_i$ with the arc on $\eta_{i-1}$ joining $p_{i-1}$ and $p_i$ (see \autoref{fig:comes-with-flip}).
  \begin{figure}[h]
    \centering
    \def\svgscale{0.75}
    %% Creator: Inkscape inkscape 0.92.5, www.inkscape.org
%% PDF/EPS/PS + LaTeX output extension by Johan Engelen, 2010
%% Accompanies image file 'comes-with-flip.pdf' (pdf, eps, ps)
%%
%% To include the image in your LaTeX document, write
%%   \input{<filename>.pdf_tex}
%%  instead of
%%   \includegraphics{<filename>.pdf}
%% To scale the image, write
%%   \def\svgwidth{<desired width>}
%%   \input{<filename>.pdf_tex}
%%  instead of
%%   \includegraphics[width=<desired width>]{<filename>.pdf}
%%
%% Images with a different path to the parent latex file can
%% be accessed with the `import' package (which may need to be
%% installed) using
%%   \usepackage{import}
%% in the preamble, and then including the image with
%%   \import{<path to file>}{<filename>.pdf_tex}
%% Alternatively, one can specify
%%   \graphicspath{{<path to file>/}}
%% 
%% For more information, please see info/svg-inkscape on CTAN:
%%   http://tug.ctan.org/tex-archive/info/svg-inkscape
%%
\begingroup%
  \makeatletter%
  \providecommand\color[2][]{%
    \errmessage{(Inkscape) Color is used for the text in Inkscape, but the package 'color.sty' is not loaded}%
    \renewcommand\color[2][]{}%
  }%
  \providecommand\transparent[1]{%
    \errmessage{(Inkscape) Transparency is used (non-zero) for the text in Inkscape, but the package 'transparent.sty' is not loaded}%
    \renewcommand\transparent[1]{}%
  }%
  \providecommand\rotatebox[2]{#2}%
  \newcommand*\fsize{\dimexpr\f@size pt\relax}%
  \newcommand*\lineheight[1]{\fontsize{\fsize}{#1\fsize}\selectfont}%
  \ifx\svgwidth\undefined%
    \setlength{\unitlength}{360.79819961bp}%
    \ifx\svgscale\undefined%
      \relax%
    \else%
      \setlength{\unitlength}{\unitlength * \real{\svgscale}}%
    \fi%
  \else%
    \setlength{\unitlength}{\svgwidth}%
  \fi%
  \global\let\svgwidth\undefined%
  \global\let\svgscale\undefined%
  \makeatother%
  \begin{picture}(1,1.01421611)%
    \lineheight{1}%
    \setlength\tabcolsep{0pt}%
    \put(0,0){\includegraphics[width=\unitlength,page=1]{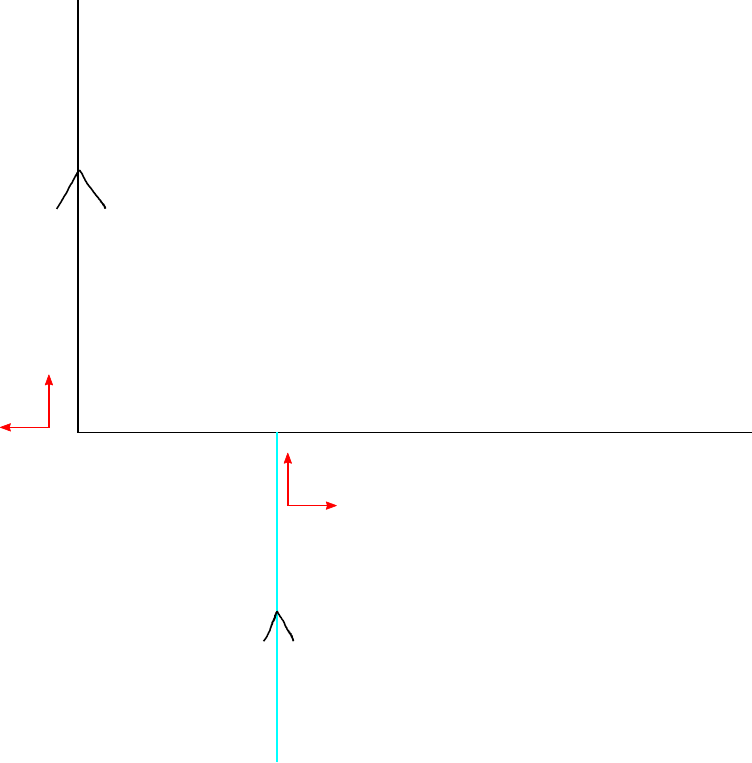}}%
    \put(0.76433602,0.45324598){\color[rgb]{0,0,0}\makebox(0,0)[lt]{\lineheight{1.25}\smash{\begin{tabular}[t]{l}$\eta_0$\end{tabular}}}}%
    \put(0.17862833,0.4591413){\color[rgb]{0,0,0}\makebox(0,0)[lt]{\lineheight{1.25}\smash{\begin{tabular}[t]{l}$\eta_i$\end{tabular}}}}%
    \put(0.35188583,0.45132822){\color[rgb]{0,0,0}\makebox(0,0)[lt]{\lineheight{1.25}\smash{\begin{tabular}[t]{l}$p_i$\end{tabular}}}}%
    \put(0.05913029,0.40041212){\color[rgb]{0,0,0}\makebox(0,0)[lt]{\lineheight{1.25}\smash{\begin{tabular}[t]{l}$p_0$\end{tabular}}}}%
    \put(0.38023604,0.0856069){\color[rgb]{0,0,0}\makebox(0,0)[lt]{\lineheight{1.25}\smash{\begin{tabular}[t]{l}$a_i$\end{tabular}}}}%
    \put(0,0){\includegraphics[width=\unitlength,page=2]{comes-with-flip.pdf}}%
    \put(0.56470636,0.39584689){\color[rgb]{0,0,0}\makebox(0,0)[lt]{\lineheight{1.25}\smash{\begin{tabular}[t]{l}$p_{i-1}$\end{tabular}}}}%
    \put(0,0){\includegraphics[width=\unitlength,page=3]{comes-with-flip.pdf}}%
  \end{picture}%
\endgroup%

    \caption{The curve $c_i$ is colored blue. Since the leaf from $p_0$ returns with the local orientation flipped to both $p_{i-1}$ and $p_{i}$, the curve $c_i$ is two-sided.}
    \label{fig:comes-with-flip}
  \end{figure}
  We have that as $i$ goes to $\infty$, the length of $c_i$ must go to $\infty$ as well, otherwise a subsequence would converge to a closed vertical curve starting at $p_0$, which cannot happen since the leaf through $p_0$ equidistributes.
  Also, note that the average of the curves $c_i$ weighted by their lengths for $i^{\prime} < i < i^{\prime \prime}$ where $i^{\prime \prime} \gg i^{\prime}$ is close to the ergodic measure, since we assumed that the leaf through $p_0$ equidistributes.
  This lets us invoke \autoref{lem:convexity-argument} to claim that there is a subsequence of $c_i$ whose orbit measures converge to the ergodic measure.
  Consequently, the geodesic representatives of $c_i$ converge to $\lambda_j$, since the geodesic tightening is close to the original curve, by the virtue of the initial and final tangent vectors being arbitrarily close.
  This resolves the two cases that can appear in the case of an orientable foliation, proving the result for case (ii).

  For case (iii), we define the points $p_i$, the nested intervals $\eta_i$, and the arcs $a_i$ in a similar manner as to case (ii).
  The key difference is that we no longer have that the foliation is orientable, which means the leaf can approach $p_i$ in one of four possible ways: from the ``up'' or the ``down'' direction, and with or without the orientation flipped.

  In case that the leaf approaches $p_i$ from the ``down'' direction without the orientation flipped infinitely often, the same closing argument as case (ii) works.
  Suppose now that the leaf approaches $p_i$ from the ``up'' direction, but without the orientation flipping, infinitely often.
  We then construct simple two-sided curves by concatenating the flow with the arc joining $p_i$ to $p_0$.
  % \begin{figure}[h]
  %   \centering
  %   \includegraphics{example-image-b}
  %   \caption{Constructing the curve $c_i$ when returning from the ``up'' direction.}
  %   \label{fig:coming-from-above}
  % \end{figure}
  While this curve does equidistribute with respect to the ergodic measure, it is not necessary that its geodesic tightening will do so.
  Denote the geodesic tightening by $c_i^{\prime}$: we have that its intersection number with $\lambda_j$ goes to $0$ as $i$ goes to $\infty$.
  By the compactness of the space of transverse probability measures, we must have that $\mu_{c_i^{\prime}}$ converges to some projective measured foliation $\gamma$ which has $0$ intersection number with $\lambda_j$, but is still supported on a subset of the support of $\lambda_j$.
  This means $\gamma$ must be another projective measured foliation in the topological conjugacy class of $\lambda_j$, i.e. is supported on the same underlying foliation.
  This proves the furthermore case of theorem.
  If $\lambda_j$ is actually uniquely ergodic, there is only one measure in the simplex of invariant probability measures, namely the uniquely ergodic one, and therefore $\mu_{g_i}$ is forced to converge to it.

  Suppose now that neither of the first two scenarios occur, i.e. the leaf returns to $p_i$ from the ``up'' or ``down'' direction, but with the orientation always flipped.
  We deal with this case like we did with the second subcase of case (ii).
  See \autoref{fig:case3-surgery} for the construction of the two-sided curves $c_i$.
  \begin{figure}[h]
    \centering
    \def\svgscale{0.75}
    %% Creator: Inkscape inkscape 0.92.5, www.inkscape.org
%% PDF/EPS/PS + LaTeX output extension by Johan Engelen, 2010
%% Accompanies image file 'case-3-surgery.pdf' (pdf, eps, ps)
%%
%% To include the image in your LaTeX document, write
%%   \input{<filename>.pdf_tex}
%%  instead of
%%   \includegraphics{<filename>.pdf}
%% To scale the image, write
%%   \def\svgwidth{<desired width>}
%%   \input{<filename>.pdf_tex}
%%  instead of
%%   \includegraphics[width=<desired width>]{<filename>.pdf}
%%
%% Images with a different path to the parent latex file can
%% be accessed with the `import' package (which may need to be
%% installed) using
%%   \usepackage{import}
%% in the preamble, and then including the image with
%%   \import{<path to file>}{<filename>.pdf_tex}
%% Alternatively, one can specify
%%   \graphicspath{{<path to file>/}}
%% 
%% For more information, please see info/svg-inkscape on CTAN:
%%   http://tug.ctan.org/tex-archive/info/svg-inkscape
%%
\begingroup%
  \makeatletter%
  \providecommand\color[2][]{%
    \errmessage{(Inkscape) Color is used for the text in Inkscape, but the package 'color.sty' is not loaded}%
    \renewcommand\color[2][]{}%
  }%
  \providecommand\transparent[1]{%
    \errmessage{(Inkscape) Transparency is used (non-zero) for the text in Inkscape, but the package 'transparent.sty' is not loaded}%
    \renewcommand\transparent[1]{}%
  }%
  \providecommand\rotatebox[2]{#2}%
  \newcommand*\fsize{\dimexpr\f@size pt\relax}%
  \newcommand*\lineheight[1]{\fontsize{\fsize}{#1\fsize}\selectfont}%
  \ifx\svgwidth\undefined%
    \setlength{\unitlength}{360.79819961bp}%
    \ifx\svgscale\undefined%
      \relax%
    \else%
      \setlength{\unitlength}{\unitlength * \real{\svgscale}}%
    \fi%
  \else%
    \setlength{\unitlength}{\svgwidth}%
  \fi%
  \global\let\svgwidth\undefined%
  \global\let\svgscale\undefined%
  \makeatother%
  \begin{picture}(1,0.91978593)%
    \lineheight{1}%
    \setlength\tabcolsep{0pt}%
    \put(0,0){\includegraphics[width=\unitlength,page=1]{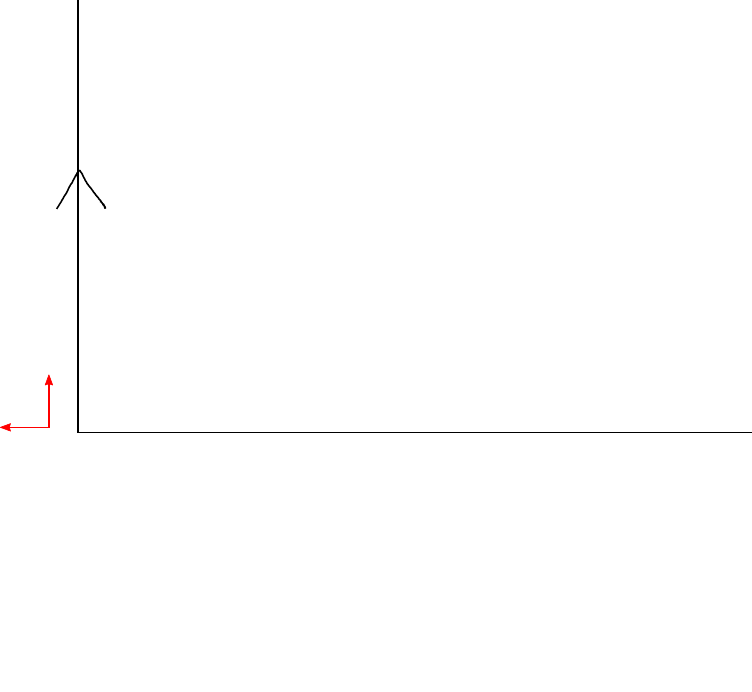}}%
    \put(0.76433602,0.3588158){\color[rgb]{0,0,0}\makebox(0,0)[lt]{\lineheight{1.25}\smash{\begin{tabular}[t]{l}$\eta_0$\end{tabular}}}}%
    \put(0.17862833,0.36471112){\color[rgb]{0,0,0}\makebox(0,0)[lt]{\lineheight{1.25}\smash{\begin{tabular}[t]{l}$\eta_i$\end{tabular}}}}%
    \put(0.05913029,0.30598194){\color[rgb]{0,0,0}\makebox(0,0)[lt]{\lineheight{1.25}\smash{\begin{tabular}[t]{l}$p_0$\end{tabular}}}}%
    \put(0,0){\includegraphics[width=\unitlength,page=2]{case-3-surgery.pdf}}%
  \end{picture}%
\endgroup%

    \caption{Construction of the blue curve $c_i$ when the leaf always returns with orientation flipped from the ``up'' or ``down'' direction.}
    \label{fig:case3-surgery}
  \end{figure}
  We have that the geodesic tightenings of the curves $c_i$ are close to the original curve by the Anosov closing lemma, and that the weighted averages of the $c_i$ converge the ergodic measure, which means by \autoref{lem:convexity-argument} we have a subsequence $\mu_{c_i}$ that converges to the ergodic measure.
  This proves the result for case (iii), and therefore the theorem.
\end{proof}

\section{Upper bound for the limit set}
\label{sec:upper-bound-limit-set}

In this section, we prove that $\geolim(\mcg(\no_g))$ is contained in $\pmf^+(\no_g)$.
We do so by defining an $\mcg(\no_g)$-invariant subset $\systole(\no_g)$, and showing that the intersection of its closure with $\pmf(\no_g)$ is contained in $\pmf^+(\no_g)$.

\begin{definition}[One-sided systole superlevel set]
  For any $\varepsilon > 0$, the set $\systole(\no_g)$ is the set of all points in $\teich(\no_g)$ where the length of the shortest one-sided curve is greater than or equal to $\varepsilon$.
\end{definition}

We can state the main theorem of this section.
\begin{theorem}
  \label{thm:systole-closure}
  For any $\varepsilon > 0$, $\overline{\systole(\no_g)} \cap \pmf(\no_g)$ is contained in $\pmf^+(\no_g)$.
\end{theorem}
The key idea of the proof is proving a quantitative estimate on the Fenchel-Nielsen coordinates of points converging to points in $\pmf^-(\no_g)$.
\begin{proposition}
  \label{prop:pinching}
  Let $\{m_i\}$ be a sequence of points in $\teich(\no_g)$ converging to a projective measured foliation $[\lambda]$.
  If $p$ is a one-sided atom of $\lambda$, for any Fenchel-Nielsen coordinate chart containing $p$ as a cuff, the length coordinate of $p$ goes to $0$.
\end{proposition}

\textit{Outline of proof.} The proof of \autoref{prop:pinching} proceeds in two steps:
\begin{enumerate}[(i)]
\item We first show that there is a curve $p_3$ intersecting $p$ such that $p_3$ is left invariant by Dehn twisting along the two-sided curve that deformation retracts onto $2p$ (when $p$ is thought of as an element in $\pi_1(\no_g)$).
  We do so in  \autoref{claim:one-sided-rigidity} and \autoref{claim:intersects-atmost-once}.
  This gives an upper bound for the length of $p_3$ in terms of the length of $p$, and an orthogeodesic going through $p$.
\item We use the upper bound obtained in the previous step to show that if the length of $p_3$ goes to $\infty$, the length of $p$ must go to $0$.
  This result can be thought of as a converse to the collar lemma, using the additional hypotheses we manage to obtain from the previous step.
\end{enumerate}

\begin{proof}[Proof of \autoref{prop:pinching}]
  Consider the following decomposition of the measured foliation $\lambda$.
  \begin{align*}
    \lambda = 1 \cdot p + \lambda_{\mathrm{at}} + \lambda_{\mathrm{Leb}}
  \end{align*}
  Here, $\lambda_{\mathrm{at}}$ are the minimal components on periodic components other than $p$, i.e. cylinders and Möbius strips, and $\lambda_{\mathrm{Leb}}$ are non-periodic minimal components.
  In the above expression, $p$ is the one-sided curve considered as a measured foliation (since we're picking a representative of $[\lambda]$, we can pick one such that $p$ has weight $1$).

  Pick simple closed curves $p_0$, $p_1$, and $p_2$, where $p_0$ is the curve $p$, and $\{p_0, p_1, p_2\}$ bound a pair of pants.
  Furthermore, we impose the following conditions on $p_1$ and $p_2$.
  \begin{align*}
    i(p_1, \lambda_{\mathrm{at}}) &= 0 \\
    i(p_2, \lambda_{\mathrm{at}}) &= 0
  \end{align*}
  Note that this can always be done, by deleting the support of $\lambda_{\mathrm{at}}$, and looking at the resulting subsurfaces.
  Neither $p_1$ nor $p_2$ can be the same as $p_0$, since $p_0$ is one-sided.

  Consider now a collection of curves $\{q\}$ which satisfy the following two constraints.
  \begin{enumerate}[(i)]
  \item $i(q, p_0) = 1$.
  \item $i(q, p_1) = 0$ and $i(q, p_2) = 0$.
  \end{enumerate}
  We use the fact that $p_0$ is one-sided to make the following claim.
  \begin{claim}
    \label{claim:one-sided-rigidity}
    There is exactly one curve $q$ up to homotopy that satisfies conditions $(i)$ and $(ii)$.
  \end{claim}
  \begin{proof}
    Let $q_1$ and $q_2$ be two curves satisfying both the conditions.
    We can assume without loss of generality that both $q_1$ and $q_2$ intersect $p_0$ at the same point.
    We now delete the curves $p_0$, $p_1$, and $p_2$ to get a pair of pants $\mathcal{P}$: denote the boundary component corresponding to $p_0$ by $\widetilde{p_0}$, and the arcs corresponding to $q_1$ and $q_2$ by $\widetilde{q_1}$ and $\widetilde{q_2}$.
    Since $p_0$ was one-sided, $\widetilde{q_1}$ and $\widetilde{q_2}$ intersect $\widetilde{p_0}$ at two points, which are diametrically opposite (with respect to the induced metric on the geodesic $\widetilde{p_0}$).

    On a pair of pants, two arcs going from a boundary component to the same component must differ by Dehn twists along that component up to homotopy relative to the boundary components: this is a consequence of the fact that the mapping class group of $\mathcal{P}$ is $\mathbb{Z}^3$, where each $\mathbb{Z}$ component is generated by a Dehn twist along a boundary component.
    This means that there is a some Dehn twist $D$ along the boundary component $\widetilde{p_0}$ such that $D\widetilde{q_1}$ is homotopic to $\widetilde{q_2}$ relative to its endpoints. Let $\widetilde{q_2}$ now denote $D\widetilde{q_1}$.

    We claim that after quotienting $\widetilde{p_0}$ by the antipodal map, $\widetilde{q_1}$ and $\widetilde{q_2}$ map to homotopic curves.
    The homotopy is obtained by moving the point of intersection of $\wt{q_2}$ and $p_0$ twice around the curve $p_0$.

    \autoref{fig:q1q2} shows the two arcs on $\mathcal{P}$ and \autoref{fig:q1toq2} shows the homotopy on the quotient that takes $\widetilde{q_2}$ to $\widetilde{q_1}$ (the movement of the blue arc is indicated by the blue arrows in \autoref{fig:q1toq2}).
    \begin{figure}[h]
      \centering
    \def\svgscale{0.6}
    %% Creator: Inkscape 1.2.1 (9c6d41e410, 2022-07-14), www.inkscape.org
%% PDF/EPS/PS + LaTeX output extension by Johan Engelen, 2010
%% Accompanies image file 'q1q2.pdf' (pdf, eps, ps)
%%
%% To include the image in your LaTeX document, write
%%   \input{<filename>.pdf_tex}
%%  instead of
%%   \includegraphics{<filename>.pdf}
%% To scale the image, write
%%   \def\svgwidth{<desired width>}
%%   \input{<filename>.pdf_tex}
%%  instead of
%%   \includegraphics[width=<desired width>]{<filename>.pdf}
%%
%% Images with a different path to the parent latex file can
%% be accessed with the `import' package (which may need to be
%% installed) using
%%   \usepackage{import}
%% in the preamble, and then including the image with
%%   \import{<path to file>}{<filename>.pdf_tex}
%% Alternatively, one can specify
%%   \graphicspath{{<path to file>/}}
%% 
%% For more information, please see info/svg-inkscape on CTAN:
%%   http://tug.ctan.org/tex-archive/info/svg-inkscape
%%
\begingroup%
  \makeatletter%
  \providecommand\color[2][]{%
    \errmessage{(Inkscape) Color is used for the text in Inkscape, but the package 'color.sty' is not loaded}%
    \renewcommand\color[2][]{}%
  }%
  \providecommand\transparent[1]{%
    \errmessage{(Inkscape) Transparency is used (non-zero) for the text in Inkscape, but the package 'transparent.sty' is not loaded}%
    \renewcommand\transparent[1]{}%
  }%
  \providecommand\rotatebox[2]{#2}%
  \newcommand*\fsize{\dimexpr\f@size pt\relax}%
  \newcommand*\lineheight[1]{\fontsize{\fsize}{#1\fsize}\selectfont}%
  \ifx\svgwidth\undefined%
    \setlength{\unitlength}{255.80665889bp}%
    \ifx\svgscale\undefined%
      \relax%
    \else%
      \setlength{\unitlength}{\unitlength * \real{\svgscale}}%
    \fi%
  \else%
    \setlength{\unitlength}{\svgwidth}%
  \fi%
  \global\let\svgwidth\undefined%
  \global\let\svgscale\undefined%
  \makeatother%
  \begin{picture}(1,1)%
    \lineheight{1}%
    \setlength\tabcolsep{0pt}%
    \put(0,0){\includegraphics[width=\unitlength,page=1]{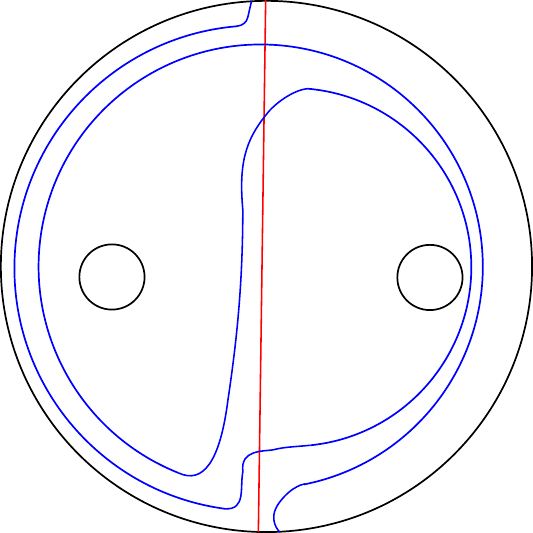}}%
    \put(0.52868411,0.44949893){\makebox(0,0)[lt]{\lineheight{1.25}\smash{\begin{tabular}[t]{l}$\wt{q_1}$\end{tabular}}}}%
    \put(0.37587016,0.44386964){\makebox(0,0)[lt]{\lineheight{1.25}\smash{\begin{tabular}[t]{l}$\wt{q_2}$\end{tabular}}}}%
  \end{picture}%
\endgroup%

      \caption{The arcs $\widetilde{q_1}$ and $\widetilde{q_2}$.}
      \label{fig:q1q2}
    \end{figure}

    \begin{figure}[h]
      \centering
    \def\svgscale{0.5}
    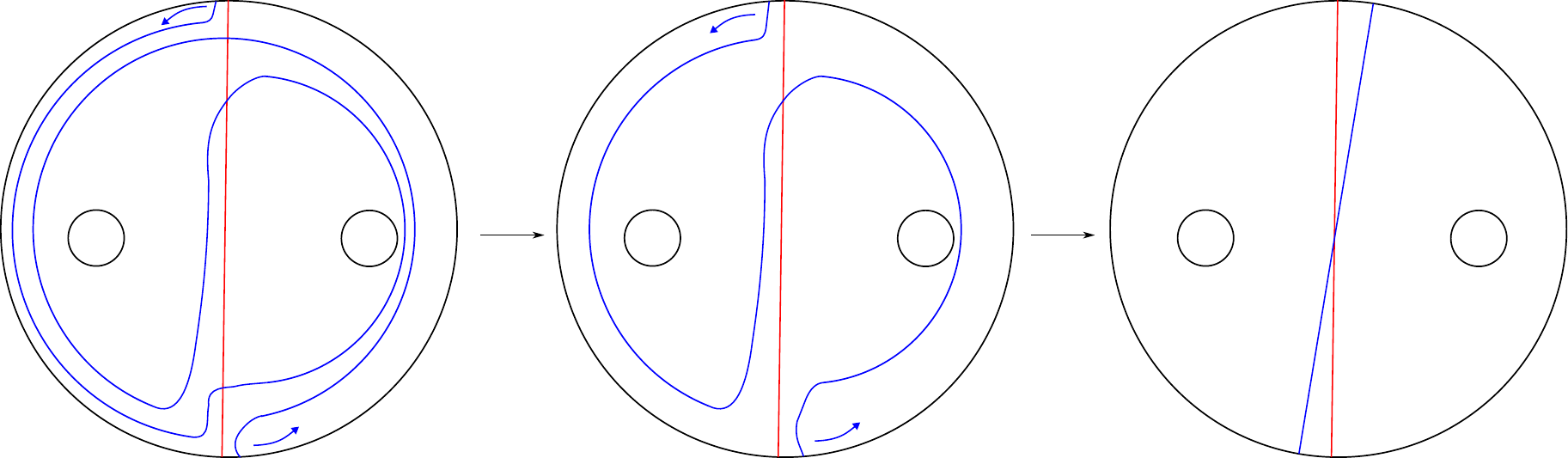

      \caption{Homotopy taking $q_2$ to $q_1$.}
      \label{fig:q1toq2}
    \end{figure}

    We have thus constructed the desired homotopy from $q_1$ to $q_2$.
    The example in \autoref{fig:q1q2} also shows there is at least one such curve, proving the claim.
  \end{proof}

  Let $p_3$ be the geodesic representative of the curve described in Claim \ref{claim:one-sided-rigidity}.
  We also define $p_4$ to be the orthogeodesic arc from $p_0$ to itself.
  We make the following claim about $p_3$ and $p_4$.

  \begin{claim}
    \label{claim:intersects-atmost-once}
    The arc $p_4$ and the curve $p_3$ intersect at most once.
  \end{claim}
  \begin{proof}
    We know from Claim \ref{claim:one-sided-rigidity} that $p_3$ is homotopic to any other curve which intersects $p_0$ exactly once and does not intersect $p_1$ and $p_2$.
    It then suffices to construct a curve $q$ that intersects $p_4$ at most once: since $p_3$ is the geodesic representative of $q$, it will also intersect $p_4$ at most once.
    We construct $q$ by starting along $p_0$, near the point where $p_4$ intersects $p_0$, and then travel parallel to $p_4$.
    When the curve reaches $p_0$ again, it will need to turn left or right to close up. In one of these cases, it will have to intersect $p_4$ once, and in the other case, it will not intersect $p_4$ at all.
  \end{proof}
  With claims \ref{claim:one-sided-rigidity} and \ref{claim:intersects-atmost-once}, we have the following picture of $\left\{ p_0, p_1, p_2, p_3, p_4 \right\}$ on the pair of pants.
  \begin{figure}[h]
    \centering
    \def\svgscale{0.75}
    %% Creator: Inkscape 1.0.1 (3bc2e813f5, 2020-09-07), www.inkscape.org
%% PDF/EPS/PS + LaTeX output extension by Johan Engelen, 2010
%% Accompanies image file 'pants.pdf' (pdf, eps, ps)
%%
%% To include the image in your LaTeX document, write
%%   \input{<filename>.pdf_tex}
%%  instead of
%%   \includegraphics{<filename>.pdf}
%% To scale the image, write
%%   \def\svgwidth{<desired width>}
%%   \input{<filename>.pdf_tex}
%%  instead of
%%   \includegraphics[width=<desired width>]{<filename>.pdf}
%%
%% Images with a different path to the parent latex file can
%% be accessed with the `import' package (which may need to be
%% installed) using
%%   \usepackage{import}
%% in the preamble, and then including the image with
%%   \import{<path to file>}{<filename>.pdf_tex}
%% Alternatively, one can specify
%%   \graphicspath{{<path to file>/}}
%%
%% For more information, please see info/svg-inkscape on CTAN:
%%   http://tug.ctan.org/tex-archive/info/svg-inkscape
%%
\begingroup%
  \makeatletter%
  \providecommand\color[2][]{%
    \errmessage{(Inkscape) Color is used for the text in Inkscape, but the package 'color.sty' is not loaded}%
    \renewcommand\color[2][]{}%
  }%
  \providecommand\transparent[1]{%
    \errmessage{(Inkscape) Transparency is used (non-zero) for the text in Inkscape, but the package 'transparent.sty' is not loaded}%
    \renewcommand\transparent[1]{}%
  }%
  \providecommand\rotatebox[2]{#2}%
  \newcommand*\fsize{\dimexpr\f@size pt\relax}%
  \newcommand*\lineheight[1]{\fontsize{\fsize}{#1\fsize}\selectfont}%
  \ifx\svgwidth\undefined%
    \setlength{\unitlength}{260.94019372bp}%
    \ifx\svgscale\undefined%
      \relax%
    \else%
      \setlength{\unitlength}{\unitlength * \real{\svgscale}}%
    \fi%
  \else%
    \setlength{\unitlength}{\svgwidth}%
  \fi%
  \global\let\svgwidth\undefined%
  \global\let\svgscale\undefined%
  \makeatother%
  \begin{picture}(1,0.88200334)%
    \lineheight{1}%
    \setlength\tabcolsep{0pt}%
    \put(0,0){\includegraphics[width=\unitlength,page=1]{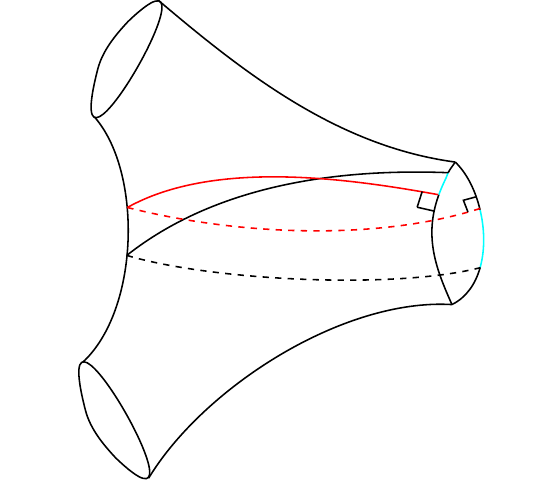}}%
    \put(0.12957868,0.81764183){\color[rgb]{0,0,0}\makebox(0,0)[lt]{\lineheight{1.25}\smash{\begin{tabular}[t]{l}$p_1$\end{tabular}}}}%
    \put(0.10381729,0.08488997){\color[rgb]{0,0,0}\makebox(0,0)[lt]{\lineheight{1.25}\smash{\begin{tabular}[t]{l}$p_2$\end{tabular}}}}%
    \put(0.86385633,0.30507798){\color[rgb]{0,0,0}\makebox(0,0)[lt]{\lineheight{1.25}\smash{\begin{tabular}[t]{l}$p_0$\end{tabular}}}}%
    \put(0.32207475,0.31452543){\color[rgb]{0,0,0}\makebox(0,0)[lt]{\lineheight{1.25}\smash{\begin{tabular}[t]{l}$p_3$\end{tabular}}}}%
    \put(0.28738299,0.56952263){\color[rgb]{0,0,0}\makebox(0,0)[lt]{\lineheight{1.25}\smash{\begin{tabular}[t]{l}$p_4$\end{tabular}}}}%
  \end{picture}%
\endgroup%

    \caption{The curves restricted to a pair of pants.}
    \label{fig:pants}
  \end{figure}
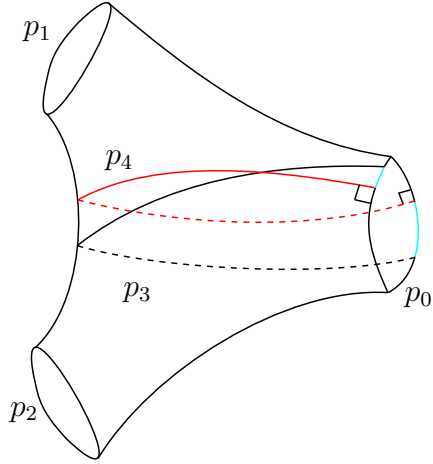

  Since $i(p_3, p_0) = 1$, and $p_0$ is a component of the limiting foliation, the length of $p_3$ must go to $\infty$.
  On the other hand, we can bound the length of $p_3$ above and below via the lengths of the orthogeodesic $p_4$ and the length of $p_0$.
 \begin{equation}
   \label{eq:up-1}
   \ell(p_3) \leq \ell(p_4) + \ell(p_0)
 \end{equation}
 Observe that the upper bound follows from \autoref{claim:intersects-atmost-once} and the fact that the red and cyan arcs are isotopic to $p_3$ relative to their endpoints being fixed.
 The cyan arcs have length at most $\ell(p_0)$ in this setting; if one allowed a twist parameter, the length of the cyan arcs would be proportional to the twist parameters.
 The point of this inequality is that we can estimate $\ell(p_4)$ using $\ell(p_0)$, $\ell(p_1)$ and $\ell(p_2)$ via hyperbolic trigonometry.
 Cut the pair of pants along the seams, to get a hyperbolic right-angled hexagon, pictured in \autoref{fig:hexagon}.
 \begin{figure}[h]
   \centering
    \def\svgscale{1}
    %% Creator: Inkscape 1.0.1 (3bc2e813f5, 2020-09-07), www.inkscape.org
%% PDF/EPS/PS + LaTeX output extension by Johan Engelen, 2010
%% Accompanies image file 'hexagon.pdf' (pdf, eps, ps)
%%
%% To include the image in your LaTeX document, write
%%   \input{<filename>.pdf_tex}
%%  instead of
%%   \includegraphics{<filename>.pdf}
%% To scale the image, write
%%   \def\svgwidth{<desired width>}
%%   \input{<filename>.pdf_tex}
%%  instead of
%%   \includegraphics[width=<desired width>]{<filename>.pdf}
%%
%% Images with a different path to the parent latex file can
%% be accessed with the `import' package (which may need to be
%% installed) using
%%   \usepackage{import}
%% in the preamble, and then including the image with
%%   \import{<path to file>}{<filename>.pdf_tex}
%% Alternatively, one can specify
%%   \graphicspath{{<path to file>/}}
%%
%% For more information, please see info/svg-inkscape on CTAN:
%%   http://tug.ctan.org/tex-archive/info/svg-inkscape
%%
\begingroup%
  \makeatletter%
  \providecommand\color[2][]{%
    \errmessage{(Inkscape) Color is used for the text in Inkscape, but the package 'color.sty' is not loaded}%
    \renewcommand\color[2][]{}%
  }%
  \providecommand\transparent[1]{%
    \errmessage{(Inkscape) Transparency is used (non-zero) for the text in Inkscape, but the package 'transparent.sty' is not loaded}%
    \renewcommand\transparent[1]{}%
  }%
  \providecommand\rotatebox[2]{#2}%
  \newcommand*\fsize{\dimexpr\f@size pt\relax}%
  \newcommand*\lineheight[1]{\fontsize{\fsize}{#1\fsize}\selectfont}%
  \ifx\svgwidth\undefined%
    \setlength{\unitlength}{267.66646498bp}%
    \ifx\svgscale\undefined%
      \relax%
    \else%
      \setlength{\unitlength}{\unitlength * \real{\svgscale}}%
    \fi%
  \else%
    \setlength{\unitlength}{\svgwidth}%
  \fi%
  \global\let\svgwidth\undefined%
  \global\let\svgscale\undefined%
  \makeatother%
  \begin{picture}(1,0.57003547)%
    \lineheight{1}%
    \setlength\tabcolsep{0pt}%
    \put(0,0){\includegraphics[width=\unitlength,page=1]{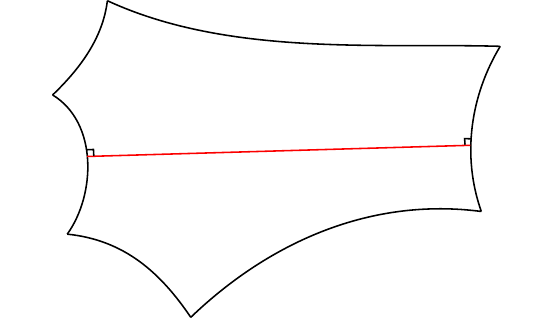}}%
    \put(0.42584046,0.32337803){\color[rgb]{0,0,0}\makebox(0,0)[lt]{\lineheight{1.25}\smash{\begin{tabular}[t]{l}$\frac{\ell(p_4)}{2}$\end{tabular}}}}%
    \put(0.08124051,0.4990387){\color[rgb]{0,0,0}\makebox(0,0)[lt]{\lineheight{1.25}\smash{\begin{tabular}[t]{l}$\frac{\ell(p_1)}{2}$\end{tabular}}}}%
    \put(0.12344631,0.06853324){\color[rgb]{0,0,0}\makebox(0,0)[lt]{\lineheight{1.25}\smash{\begin{tabular}[t]{l}$\frac{\ell(p_2)}{2}$\end{tabular}}}}%
    \put(0.87404692,0.39677773){\color[rgb]{0,0,0}\makebox(0,0)[lt]{\lineheight{1.25}\smash{\begin{tabular}[t]{l}$f \cdot \ell(p_0)$\end{tabular}}}}%
    \put(0.86088542,0.23219518){\color[rgb]{0,0,0}\makebox(0,0)[lt]{\lineheight{1.25}\smash{\begin{tabular}[t]{l}$(1-f) \cdot \ell(p_0)$\end{tabular}}}}%
  \end{picture}%
\endgroup%

   \caption{The right angled hexagon obtained by cutting the pants along the seams.}
   \label{fig:hexagon}
 \end{figure}
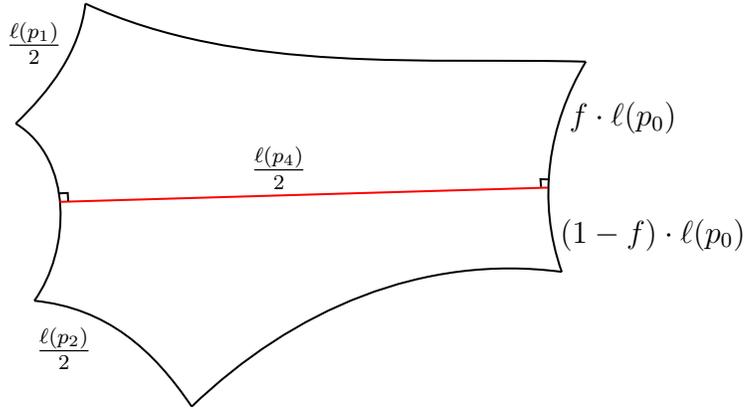

 To get good estimates on $\ell(p_4)$, we need a universal lower bound on the fraction $f$ as we move in the Teichm\"uller space.
 The analysis splits up into two cases, but it is not {a priori} clear that these two cases are exhaustive.
 We will deal with the two cases, and then show that any other case can be reduced to the second case by changing $p_1$ and $p_2$.

\subsection*{Case I}
\label{case1}
We're in this case if $p_1$ and $p_2$ don't intersect the foliation $\lambda$ at all.
\begin{align*}
  i(p_1, \lambda) &= 0 \\
  i(p_2, \lambda) &= 0
\end{align*}
In this case, we can pass to a subsequence of $\{m_i\}$ such that the corresponding values of $f$ are always greater than $\frac{1}{2}$ or less than $\frac{1}{2}$.
In the former case, we focus on $p_1$, and in the latter case, we focus on $p_2$.
Without loss of generality, we'll suppose $f \geq \frac{1}{2}$.
In that case, we cut along the orthogeodesic $p_4$, and get a hyperbolic right-angled pentagon, which is the top half of \autoref{fig:hexagon}.

Let $\ell_i(p_k)$ denote the length of $p_k$ on the hyperbolic surface corresponding to $m_i$.
We can relate $\ell_i(p_0)$, $\ell_i(p_1)$, and $\ell_i(p_4)$ using the following identity for hyperbolic right-angled pentagons (see \cite{thurston1979geometry} for the proof of the identity).
\begin{align}
  \label{eq:pentagon}
  \sinh\left( f \cdot \ell_i(p_0) \right)  \cdot \sinh\left( \frac{\ell_i(p_4)}{2} \right)  = \cosh\left( \frac{\ell_i(p_1)}{2} \right)
\end{align}
Now suppose for contradiction's sake that $\ell_i(p_0)$ does not go to $0$.
Then we must have that for all $i$, $\ell_i(p_0) \geq 2\varepsilon$ for some $\varepsilon > 0$. By the lower bound on $f$, we have that the first term on the left hand side of the above expression is bounded below by $\varepsilon$.
Rearranging the terms gives us the following upper bound on $\ell_i(p_4)$.
\begin{equation}
  \label{eq:up-2}
  \ell_i(p_4) \leq 2 \cdot \sinh^{-1} \left( \frac{\cosh \left( \frac{\ell_i(p_1)}{2} \right)}{\varepsilon} \right)
\end{equation}
Using \eqref{eq:up-1} and \eqref{eq:up-2}, we get an upper bound for $\ell_i(p_3)$.
\begin{equation}
  \label{eq:up-3}
  \ell_i(p_3) \leq \ell_i(p_0) +
  2 \sinh^{-1} \left( \frac{\cosh \left( \frac{\ell_i(p_1)}{2} \right)}{\varepsilon} \right)
\end{equation}

Since $\frac{i(p_0, \lambda)}{i(p_3, \lambda)} = 0$, as $\{m_i\}$ approaches $\lambda$, the ratio of lengths of $p_0$ and $p_3$ approach $0$.
\begin{equation}
  \label{eq:up-4}
  \lim_{i \to \infty} \frac{\ell_i(p_0)}{\ell_i(p_3)} = 0
\end{equation}
Using \eqref{eq:up-3}, we have the lower bound for $\frac{\ell_i(p_0)}{\ell_i(p_3)}$.
\begin{equation}
  \label{eq:up-5}
  \frac{\ell_i(p_0)}{  \ell_i(p_0) +
  2 \sinh^{-1} \left( \frac{\cosh \left( \frac{\ell_i(p_1)}{2} \right)}{\varepsilon} \right)} \leq \frac{\ell_i(p_0)}{\ell_i(p_3)}
\end{equation}
By \eqref{eq:up-4}, the left hand side of \eqref{eq:up-5} must go to $0$, or equivalently, the following holds.
\begin{align}
  \label{eq:up-5dot5}
\lim_{i \to \infty} \frac{\ell_i(p_0)}{2 \sinh^{-1} \left( \frac{\cosh \left( \frac{\ell_i(p_1)}{2} \right)}{\varepsilon} \right)}
= 0
\end{align}
But we also have that $\ell_i(p_0) > \varepsilon$: that means the only way that the above limit is $0$ if $\ell_i(p_1)$ goes to $\infty$.
This is where the hypotheses of the Case I come in.
Since $i(p_1, \lambda)$ is $0$, the following equality must hold.
\begin{equation}
  \label{eq:up-6}
  \lim_{i \to \infty} \frac{\ell_i(p_1)}{\ell_i(p_3)} = 0
\end{equation}
This means the lower bound for $\frac{\ell_i(p_1)}{\ell_i(p_3)}$ must go to $0$.
\begin{equation}
  \label{eq:up-7}
   \lim_{i \to \infty} \frac{\ell_i(p_1)}{  \ell_i(p_0) +
  2 \sinh^{-1} \left( \frac{\cosh \left( \frac{\ell_i(p_1)}{2} \right)}{\varepsilon} \right)
  } = 0
\end{equation}
From \eqref{eq:up-5dot5}, we have the following.
\begin{equation}
  \label{eq:up-7dot5}
  \lim_{i \to \infty} \frac{\ell_i(p_1)}{  \ell_i(p_0) +
  2 \sinh^{-1} \left( \frac{\cosh \left( \frac{\ell_i(p_1)}{2} \right)}{\varepsilon} \right)
  }
  =
  \lim_{i \to \infty} \frac{\ell_i(p_1)}{
  2 \sinh^{-1} \left( \frac{\cosh \left( \frac{\ell_i(p_1)}{2} \right)}{\varepsilon} \right)
  }
\end{equation}

But as $\ell_i(p_1)$ approaches $\infty$, the right hand side of \eqref{eq:up-7dot5} approaches a non-zero constant value, which contradicts the identity in \eqref{eq:up-7}.
This contradiction means our assumption that $\ell_i(p_0)$ was bounded away from $0$ must be wrong, and thus proves the result in Case I.

\subsection*{Case II}
We're in this case if the following inequality holds.
\begin{equation}
  \label{eq:up-8}
  0 < i(p_1, \lambda) < 1
\end{equation}
The picture in this case looks similar to \autoref{fig:hexagon}.
However, we can't necessarily pass to a subsequence where $f \geq \frac{1}{2}$ (and the trick of working with $1-f$ won't work, since we know nothing about $p_2$).
This is one of the points where the hypothesis on $p_1$ comes in.
Since $\frac{i(p_2, \lambda)}{i(p_1, \lambda)}$ is finite, we must have that the ratio of lengths $\frac{\ell_i(p_2)}{\ell_i(p_1)}$ approaches some finite value as well.
The fraction $f$ is a continuous function of $\frac{\ell_i(p_2)}{\ell_i(p_1)}$, approaching $0$ only as the ratio approaches $\infty$ (this follows from the same identity as \eqref{eq:pentagon}). Since the ratio approaches a finite value, we have a positive lower bound $f_0$ for $f$.

Assuming as before that $\ell_i(p_0)$ is bounded away from $0$, and $\tau(p_0)$ bounded away from $\pm \infty$, and repeating the calculations of the previous case, we get the following two inequalities.
\begin{equation}
  \label{eq:up-9}
  \frac{\ell_i(p_1)}{\ell_i(p_3)} \geq \frac{\ell_i(p_1) }{\ell_i(p_0) +
  2 \sinh^{-1} \left( \frac{\cosh\left( \frac{\ell_i(p_1)}{2} \right)}{f_0 \varepsilon} \right)
  }
\end{equation}
\begin{equation}
  \label{eq:up-10}
  \frac{\ell_i(p_0)}{\ell_i(p_3)} \geq \frac{\ell_i(p_0)}{\ell_i(p_0) +
  2 \sinh^{-1} \left( \frac{\cosh\left( \frac{\ell_i(p_1)}{2} \right)}{f_0 \varepsilon} \right)
  }
\end{equation}
The right hand side of \eqref{eq:up-10} must approach $0$, and that forces either $\ell_i(p_1)$ or $\ell_i(p_0)$ to approach $\infty$.
But that means the right hand term of \eqref{eq:up-9} must approach $1$, which cannot happen, by the hypothesis of case II.
This means $\ell_i(p_0)$ goes to $0$, proving the result in case II.

\subsection*{Reducing to case II} Suppose now that both $p_1$ and $p_2$ have an intersection number larger than $1$ with $\lambda$.
We can modify one of them to have a small intersection number with $\lambda$.
First, we assume that $\lambda_{\mathrm{Leb}}$ is supported on a single minimal component, i.e. every leaf of $\lambda_{\mathrm{Leb}}$ is dense in the support.
We now perform a local surgery on $p_1$: starting at a point on $p_1$ not contained in the support of $\lambda_{\mathrm{Leb}}$, we follow along until we intersect $\lambda_{\mathrm{Leb}}$ for the first time.
We denote this point by $\alpha$.
We now go along $p_1$ in the opposite direction, until we hit the support of $\lambda_{\mathrm{Leb}}$ again, but rather than stopping, we keep going until the arc has intersection number $0 < \delta < 1$ with $\lambda_{\mathrm{Leb}}$.
We then go back to $\alpha$, and follow along a leaf of $\lambda_{\mathrm{Leb}}$ rather than $p_1$, until we hit the arc.
This is guaranteed to happen by the minimality of $\lambda_{\mathrm{Leb}}$.
Once we hit the arc, we continue along the arc, and close up the curve.
This gives a new simple closed curve which intersection number with $\lambda$ is at most $\delta$.
This curve is our replacement for $p_1$.
If $\lambda_{\mathrm{Leb}}$ is not minimal, we repeat this process for each minimal component. We pick $p_2$ in a manner such that $p_0$, $p_1$, and $p_2$ bound a pair of pants.
Since $\delta < 1$, we have reduced to case II.
This concludes the proof of the theorem.
\end{proof}

\begin{remark}[On the orientable version of \autoref{prop:pinching}]
  The same idea also works in the orientable setting, although the analysis of the various cases gets a little more delicate.
  The first change one needs to make is in the statement of the proposition: we no longer need to require $p$ to be a one-sided atom, and correspondingly, either the length coordinate $\ell_i(p)$ can go to $0$, or the twist coordinate $\tau(p_0)$ can go to $\pm \infty$.
  To see how the twist coordinate enters the picture, observe that \eqref{eq:up-1}, which was the main inequality of the proof, turns into the following in the orientable version.
  \begin{equation}
    \label{eq:up-11}
    \ell_i(p_4)  \leq \ell_i(p_3) \leq \tau(p_0) + \ell_i(p_4)
  \end{equation}
  Here, $\tau(p_0)$ is the twist parameter about $p_0$, and $p_4$ is the orthogeodesic multi-arc (there may be one or two orthogeodesics, depending on the two cases described below).

  The proof splits up into two cases, depending on whether both sides of $p$ are the same pair of pants, or distinct pairs of pants.
  This was not an issue in the non-orientable setting, since $p$ was one-sided.
  If both sides of $p$ are the same pair of pants, then the analysis is similar to what we just did, since the curve $p_3$ stays within a single pair of pants.
  In the other, $p_3$ goes through two pair of pants, and its length is a function of the twist parameters, as well the cuff lengths of four curves, rather than two curves, the four curves being the two remaining cuffs of each pair of pants.
  The analysis again splits up into two cases, depending on the intersection number of the cuffs with $\lambda$, but reducing all the other cases to case II becomes tricky because we need to simultaneously reduce the intersection number of two curves, rather than one, as in the non-orientable setting.
  This added complication obscures the main idea of the proof, which is why we chose to only prove the non-orientable version.
\end{remark}

This quantitative estimate of \autoref{prop:pinching} gives us a proof for \autoref{thm:systole-closure}.
\begin{proof}[Proof of \autoref{thm:systole-closure}]
  Suppose that the theorem were false, and there was a foliation $[\lambda] \in \pmf^-(\no_g)$ in the closure of $\systole(\no_g)$.
  Suppose $p$ is a one-sided atom in $\lambda$.
  Then \autoref{prop:pinching} tells us that the hyperbolic length of $p$ goes to $0$, but the length of $p$ must be greater than $\varepsilon$ in $\systole(\no_g)$.
  This contradicts our initial assumption, and the closure of $\systole(\no_g)$ can only intersect $\pmf(\no_g)$ in the complement of $\pmf^-(\no_g)$.
\end{proof}

\begin{corollary}
  \label{cor:geolimset}
  The geometric limit set $\geolim(\mcg(\no_g))$ is contained in $\pmf^+(\no_g)$.
\end{corollary}
\begin{proof}
  Every point $p \in \teich(\no_g)$ is contained in $\systole(\no_g)$ for some small enough $\varepsilon$.
  This means $\Lambda_{\mathrm{geo}, p}(\mcg(\no_g))$ is contained in $\pmf^+(\no_g)$ by \autoref{thm:systole-closure}.
  Taking the union over all $p$ proves the result.
\end{proof}

\section{Failure of quasi-convexity for ${\systole}$}
\label{sec:fail-quasi-conv}

In the setting of Teichm\"uller geometry, convexity is usually too strong of a requirement.
For instance, metric balls in Teichm\"uller space are not convex, but merely quasi-convex (see \autocite{lenzhen2011length}).
\begin{definition}[Quasi-convexity]
  A subset $S$ of ${\teich(S)}$ is said to be quasi-convex if there is some uniform constant $D > 0$ such that the geodesic segment joining any pair of points in $S$ stays within distance $D$ of $S$.
\end{definition}

Our goal for this section will be to prove the following theorem.
\begin{theorem}
  \label{thm:qc-fail}
  For $g \geq 8$, any $\varepsilon > 0$, and all $D > 0$, there exists a Teichm\"uller geodesic segment whose endpoints lie in ${\systole}(\no_g)$ such that some point in the interior of the geodesic is more than distance $D$ from $\systole$.
\end{theorem}

\begin{remark}
  Our methods actually prove the result for all non-orientable hyperbolic surfaces except genus $5$ and $7$.
  This is not because genus $5$ and $7$ are special, but it is rather an artifact of our construction.
  We construct two families of counterexamples, one for genera $4+2j$, and one for genera $9+2j$: it turns out there isn't enough ``room'' on a genus $5$ surface to replicate our genus $9$ construction, but it's quite likely an alternate construction will work.
\end{remark}

We begin by finding Teichmüller geodesic segments whose endpoints lie in $\systole$ such that at a point in the interior, some one-sided curve gets very short.
Once we have arbitrarily short one-sided curves in the interior of the geodesic segments, estimates relating Teichmüller distance and ratios of hyperbolic lengths of curves will give us the result.
\begin{proposition}
  \label{prop:very-short-curves}
  For all $g \geq 8$ and any $\delta > 0$, there exists a Teichmüller geodesic segment $l$ whose endpoints lie in $\systole(\no_g)$, and a point $p$ in $l$ such that some one-sided curve has length less than $\delta$ with respect to the hyperbolic metric on $p$.
\end{proposition}

To prove this result, we will need two lemmas relating hyperbolic and flat lengths.
\begin{lemma}
  \label{lem: relating-flat-hyperbolic}
  Let $q$ be any area $1$ DQD on $\no_g$, and let $\gamma$ be a simple closed curve of $q$.
  Suppose that $\lhyp(\gamma) \leq \delta$ (with respect to the unique hyperbolic metric coming from the flat structure $q$).
  Then $\lflat(\gamma) \leq k \sqrt{\delta}$, where $k$ is some absolute constant.
\end{lemma}
\begin{proof}[Sketch of proof]
  If $\lhyp(\gamma) \leq \delta$, then there exists an annulus around $\gamma$ of modulus proportional to $\frac{1}{\delta}$.
  By the results in \cite{Minsky1992HarmonicML}, this annulus can be homotoped to be a primitive annulus, i.e. an annulus that does not pass through a singularity of the flat metric.
  Such annuli are either expanding, i.e. concentric circles in the flat metric, or flat, and in either case, we have an upper bound on the flat length of the core curve in terms of the modulus.
  This proves the result.
\end{proof}

\begin{lemma}
  \label{lem:schwarz}
  Let $q$ be an area $1$ DQD on $\no_g$, and consider the unique hyperbolic metric with the same conformal structure.
  Let $A$ be a primitive annulus in $q$, i.e. an annulus whose interior does not pass  through a singularity of the flat metric.
  Let the modulus of $A$ be $m$.
  Then the hyperbolic length of the isotopy class of the core curve of the annulus is at most $\frac{\pi}{m}$.
\end{lemma}
\begin{proof}[Sketch of proof]
  Without loss of generality, we can pass to the orientable double cover.
  This changes the hyperbolic lengths by at most a factor of two.
  Consider the interior of the annulus as a Riemann surface, and put the unique hyperbolic metric on that surface.
  With respect to this hyperbolic metric, the length of the core curve is $\frac{\pi}{m}$. Since the interior doesn't contain any singularities, the inclusion map is holomorphic, and holomorphic maps are distance reducing with respect to the hyperbolic metric.
  This proves the result.
\end{proof}

To find a geodesic segment whose endpoints lie in $\systole$, we will construct a DQD $q$, and use \autoref{lem: relating-flat-hyperbolic} to find large enough $t$ such that both $g_t(q)$ and $g_{-t}(q)$ are in $\systole$.
We will then show that some one sided curve on $q$ is very short using \autoref{lem:schwarz}, which will prove \autoref{prop:very-short-curves}.

\begin{proof}[Proof of \autoref{prop:very-short-curves}]
  We will prove the result by constructing explicit examples in genus $4$ and $9$, and then connect summing orientable surfaces of genus $j$ to get examples in genus $4+2j$ and $9+2j$.

  We first list the two properties we require from the DQD $q$ we want to construct, and show that having those properties proves the result.
  \begin{enumerate}[(a)]
  \item There exists an embedded annulus in $q$ with a very large modulus whose core curve is the square of a one-sided curve in $\pi_1(\no_g)$.
  \item The vertical and horizontal foliations decompose as a union of cylinders, i.e. the vertical and horizontal flow is periodic, and no closed orbit is a one-sided curve. Furthermore, deleting the core curves of the cylinders in the horizontal or vertical direction result in a disjoint union of \emph{orientable} subsurfaces.
  \end{enumerate}

We now show why having these two properties proves the result.
Suppose we have a DQD $q$ satisfying (a) and (b).
\autoref{lem:schwarz} tells us that satisfying (a) means that the one-sided curve whose square is the core curve of the annulus will be very short.
To find a large enough $t$ such that $g_t(q)$ has no one-sided curves shorter than $\varepsilon$, pick a $t$ enough such that each vertical cylinder in $g_t(q)$ is at least $2k \sqrt{\varepsilon}$ wide.
Consider now any closed curve who flat length is less than $k \sqrt{\varepsilon}$.
It must either be homotopic to one of the core curves of the vertical cylinders, or can be homotoped to be completely contained in one of the subsurfaces obtained by deleting all the core curves.
That is because it was neither of these cases, it would cross at least one of these cylinders, and since the cylinders are at least $2k \sqrt{\varepsilon}$ wide, the flat length of the curve would exceed $k \sqrt{\varepsilon}$.
If the curve is the core curve of a cylinder, or completely contained in one of the subsurfaces, it must be two-sided, by condition (b).

This proves that all one-sided curves have flat length exceeding $k\sqrt{\varepsilon}$, and therefore hyperbolic length exceeding $\varepsilon$.
The same argument also works for $g_{-t}(q)$, proving the result.

We now construct explicitly the DQDs satisfying conditions (a) and (b) in genus $4$, $9$, and above.

\subsection*{The $g=4$ case}

Consider the area $1$ DQD on $\no_4$ depicted in \autoref{fig:genus-4-example}.
\begin{figure}[h]
  \centering
    \def\svgscale{0.45}
    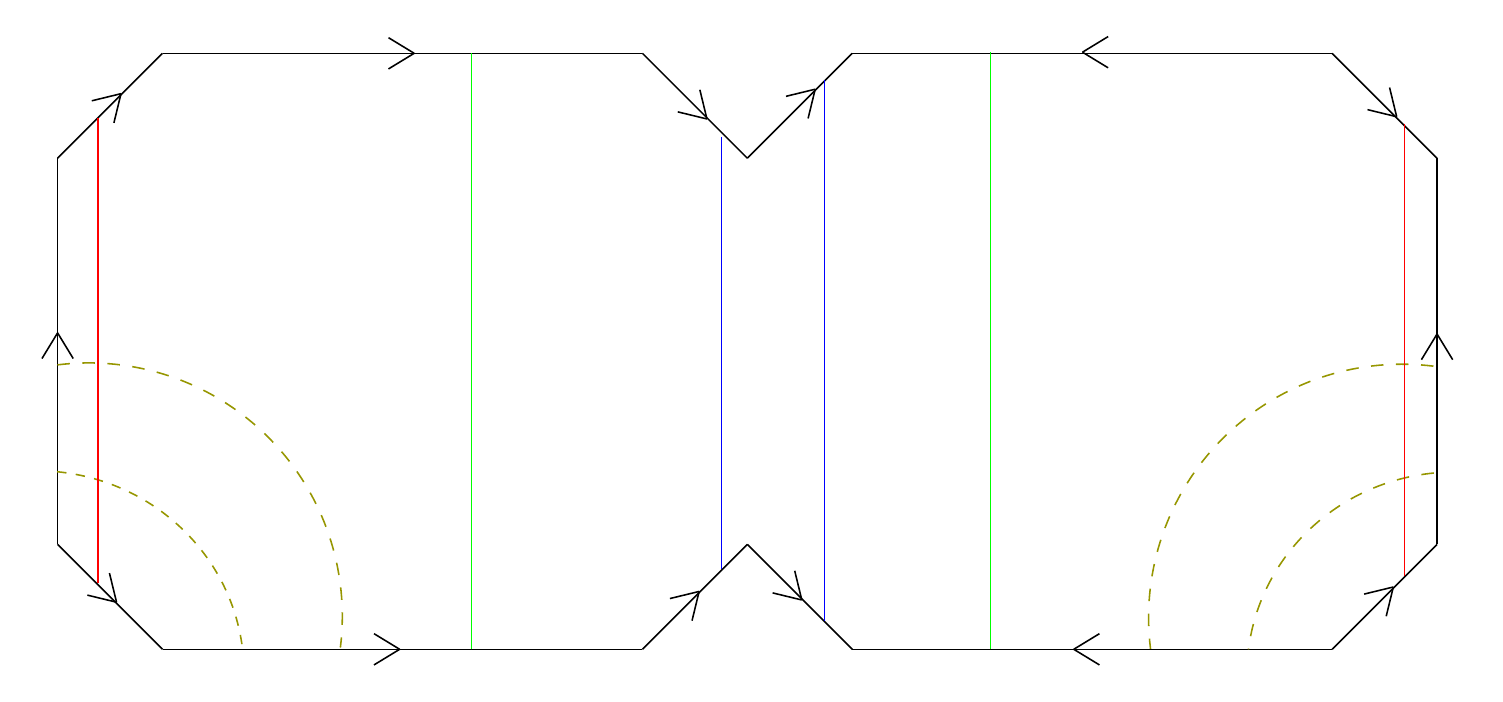

  \caption{A DQD on $\no_4$.}
  \label{fig:genus-4-example}
\end{figure}
We impose the following constraint on the depicted DQD: the edges $\{c, c^{\prime}, d, d^{\prime}\}$ are all oriented at an angle of $\pm \frac{\pi}{4}$, and have the same length.

Observe that by making the length of $c$ (and correspondingly $c^{\prime}$, $d$, and $d^{\prime}$) go to $0$, while keeping the area $1$ lets us embed an annulus of high modulus (pictured as dotted semi circle in \autoref{fig:genus-4-example}) around any curve in $\{c, c^{\prime}, d, d^{\prime}\}$.
This shows that the DQD we constructed satisfies condition (a).

Checking condition (b) is easy, but tedious.
For convenience, we have labelled the core curves of the vertical cylinders in red, blue, and green: the reader can check that they are all two-sided, and deleting them results in orientable subsurfaces.
In fact, deleting the core curves results in $2$ pairs of pants.

\subsection*{The $g=9$ case}
Consider the area $1$ DQD on $\no_9$ depicted in \autoref{fig:genus-9-example}.
To keep the picture from getting cluttered, we describe the edge gluing maps in words: the edges labelled $c$ are glued via the map $z \mapsto -\overline{z} + k$, the edges labelled $b$ and $e$ are glued via $z \mapsto -z + k$, where $k$ is some constant.
All the other gluings are translation gluings.
\begin{figure}[h]
  \centering
    \def\svgscale{0.45}
    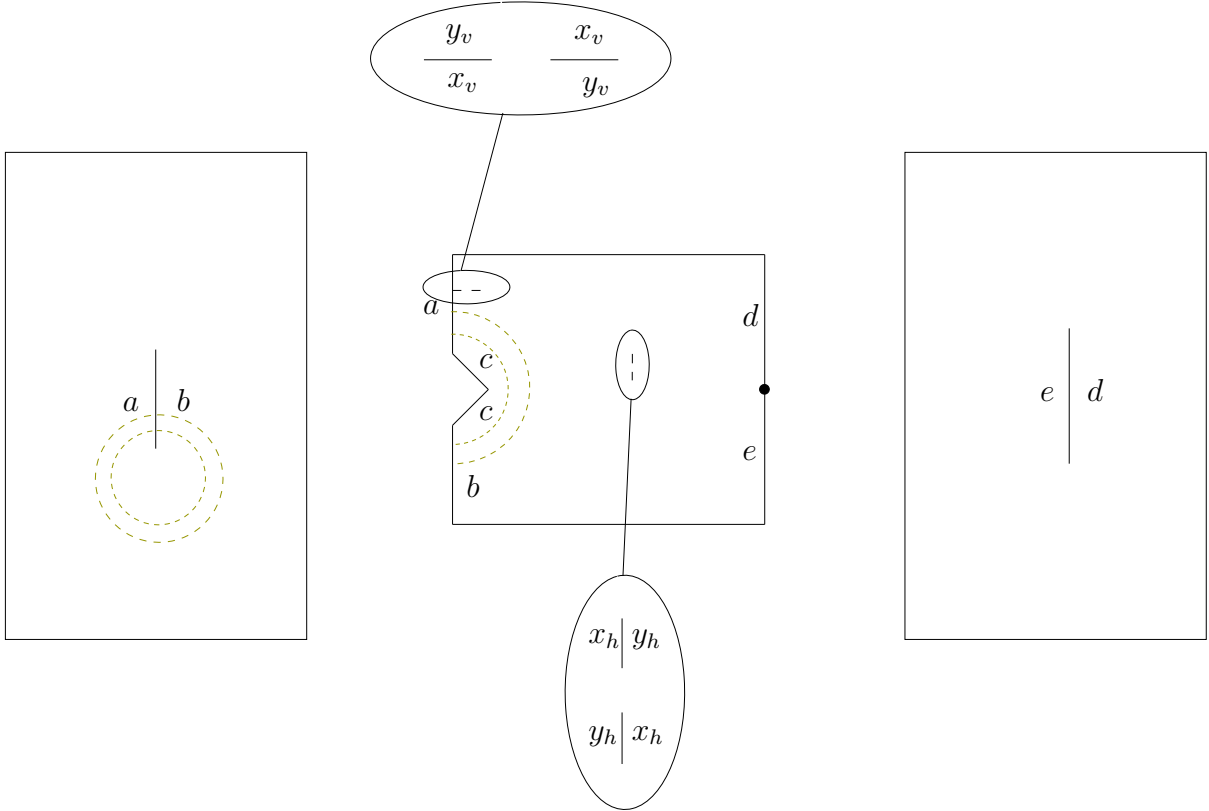

  \caption{A DQD on $\no_9$. To display the gluing maps on the small slits, we have a zoomed in picture in the ellipses.}
  \label{fig:genus-9-example}
\end{figure}
We impose the following constraints on the DQD.
\begin{enumerate}[(i)]
\item The edges labelled $c$ are oriented at an angle of $\pm \frac{\pi}{4}$, and the lengths of $\{x_h, y_h, x_v, y_v\}$ are $\frac{\lflat(c)}{4\sqrt{2}}$.
\item The left edge of $x_v$ is aligned with the left edge of $c$, the left edge of $y_v$ is aligned with the midpoint of $c$, the top edge of $x_h$ is aligned with the top edge of $c$, and the top edge of $y_h$ is aligned with the midpoint of $c$.
\end{enumerate}
By making $c$ smaller, while keeping the area equal to $1$, one can embed an annulus of high modulus in the DQD, pictured in dotted olive green in \autoref{fig:genus-9-example}.
This shows that our construction satisfies condition (a).

To see that deleting the core curves of the horizontal cylinders results in orientable subsurfaces, note that deleting the core curves passing through $c$ results in $2$ pairs of pants, and a genus $3$ orientable surface with one boundary component.
This is again easy, but tedious to verify, so we leave the verification to the reader.
This shows that the example satisfies condition (b).

\subsection*{The induction step}
To get higher genus DQDs satisfying conditions (a) and (b), we start with the $g=4$ and $g=9$ examples and connect-sum an orientable surface using the slit construction.
To ensure that the new surfaces still satisfy conditions (a) and (b), we need to ensure that the slit we construct if far away from the annulus of condition (a), as well as all the vertical and horizontal leaves passing through $\{c, c^{\prime}, d, d^{\prime}\}$ in the $g=4$ example, and the vertical and horizontal leaves passing through $c$ in the $g=9$ example.
This will ensure that the resulting higher genus surface still satisfies conditions (a) and (b).
\end{proof}

To relate Teichm\"uller distance to hyperbolic lengths, we need Wolpert's lemma (\autocite{wolpert1979length})
\begin{lemma}[Wolpert's Lemma]
  Let $M$ and $M^{\prime}$ be two points in $\teich(\os_g)$, and let $\gamma$ be a simple closed curve on $\os_g$.
  Let $R$ be the Teichm\"uller distance between $M$ and $M^{\prime}$. Then the ratio of the hyperbolic length
  of $\gamma$ and $R$ are related by the following inequalities.
  \begin{align*}
    \exp(-2R) \leq \frac{\ell_{\mathrm{hyp}}(M, \gamma)}{\ell_{\mathrm{hyp}}(M^{\prime}, \gamma)} \leq \exp(2R)
  \end{align*}
\end{lemma}

Using Proposition \ref{prop:very-short-curves} and Wolpert's lemma, we can prove Theorem
\ref{thm:qc-fail}.
\begin{proof}[Proof of Theorem \ref{thm:qc-fail}]
  Suppose that $\systole(\no_g)$ was indeed quasi-convex.
  That would mean that there exists some $R > 0$, depending on $\varepsilon$ such that every point in the interior of any geodesic segment with endpoints in $\systole$ was within $R$
  distance of some point in $\systole(\no_g)$.
  Proposition \ref{prop:very-short-curves} lets us construct a sequence of Teichmüller geodesic segments such that for on some interior point, the length of a given one-sided curve $\gamma$ goes to $0$.
  If those points were within distance $R$ of $\systole$, there would be some point in $\systole$ where the length of $\gamma$ was at most $\exp(2R)$ times the length of $\gamma$ in the geodesic, by Wolpert's lemma.
  But since the length of $\gamma$ in the geodesic goes to $0$, the length in the corresponding closest point in $\systole$ must also go to $0$.
  This violates the definition of $\systole$, giving us a contradiction, and proving the result.
\end{proof}

\printbibliography

\end{document}